\documentclass[A4]{amsart}
\usepackage[square,sort,comma,numbers]{natbib}
\usepackage{amsmath, amssymb, amstext, amsfonts, textcomp, amsxtra, amsbsy, amsgen, amsopn, amscd, mathrsfs, amsthm, latexsym, array}
\usepackage[textwidth=16cm,textheight=22cm,centering]{geometry}

\usepackage{color}
\usepackage{graphicx}

\usepackage[all]{xy}

\newtheorem{theorem}             {Theorem}  [section]
\newtheorem{definition} [theorem] {Definition}
\newtheorem{lemma}      [theorem]{Lemma}
\newtheorem{corollary}  [theorem]{Corollary}
\newtheorem{proposition}[theorem]{Proposition}
\newtheorem{remark} [theorem] {Remark}

\numberwithin{equation}{section} \everymath{\displaystyle}

\newcommand{\Cont}{{\rm C}}

\newcommand{\Sob}{{\rm S}}
\newcommand{\Sch}{\mathcal{S}}
\newcommand{\sgn}{{\rm sgn}}
\newcommand{\intL}{{\rm L}}

\newcommand{\Nr}{{\rm Nr}}

\newcommand{\gp}[1]{\mathbf{#1}}
\newcommand{\GL}{{\rm GL}}
\newcommand{\PGL}{{\rm PGL}}
\newcommand{\SL}{{\rm SL}}
\newcommand{\SO}{{\rm SO}}
\newcommand{\SU}{{\rm SU}}
\newcommand{\grG}{\mathbf{G}}

\newcommand{\grZ}{\mathbf{Z}}

\newcommand{\grB}{\mathbf{B}}

\newcommand{\grN}{\mathbf{N}}

\newcommand{\ag}[1]{\mathbb{#1}}
\newcommand{\Casimir}{\mathcal{C}}

\newcommand{\F}{\mathbf{F}}
\newcommand{\vo}{\mathfrak{o}}
\newcommand{\vp}{\mathfrak{p}}
\newcommand{\idlN}{\mathfrak{N}}

\newcommand{\Dis}{{\rm D}}

\newcommand{\Proj}{{\rm P}}
\newcommand{\norm}[1][\cdot]{\lvert #1 \rvert}
\newcommand{\extnorm}[1]{\left\lvert #1 \right\rvert}
\newcommand{\Norm}[1][\cdot]{\lVert #1 \rVert}
\newcommand{\extNorm}[1]{\left\lVert #1 \right\rVert}

\newcommand{\Pairing}[2]{\langle #1, #2 \rangle}
\newcommand{\extPairing}[2]{\left\langle #1, #2 \right\rangle}

\newcommand{\Four}[2][]{\mathfrak{F}_{#1}(#2)}

\newcommand{\Mellin}[2][]{\mathfrak{M}_{#1}(#2)}

\newcommand{\Bas}{\mathcal{B}}

\newcommand{\Res}{{\rm Res}}
\newcommand{\Ind}{{\rm Ind}}

\newcommand{\Cond}{\mathbf{C}}
\newcommand{\cond}{\mathfrak{c}}

\newcommand{\fin}{{\rm fin}}
\newcommand{\eis}{{\rm E}}

\newcommand{\eisCst}{{\rm E}_{\grN}}

\makeatletter

\newcommand{\Rmnum}[1]{\expandafter\@slowromancap\romannumeral #1@}
\makeatother
\title{Explicit Burgess-like subconvex bounds for $\GL_2 \times \GL_1$}
\author{Han Wu}
\begin{document}
	
	\begin{abstract}
		We make the polynomial dependence on the fixed representation $\pi$ in our previous subconvex bound of $L(1/2,\pi \otimes \chi)$ for $\GL_2 \times \GL_1$ explicit, especially in terms of the usual conductor $\Cond(\pi_{\fin})$. There is no clue that the original choice, due to Michel \& Venkatesh, of the test function at the infinite places should be the optimal one. Hence we also investigate a possible variant of such local choices in some special situations.
	\end{abstract}
	
	\maketitle
	
	\tableofcontents

\section{Introduction}

	\subsection{Michel \& Venkatesh's Method}
	
	In the late 1980's, Iwaniec \cite{Iw87} invented the method of amplification, which was subsequently developed by him and his collaborators \cite{Du88, DFI93, DFI94, DFI01}. The principle of this method can be abstracted as follows. Suppose that some quantities $a(\pi)$ indexed by a family $\pi \in \mathcal{F}$ admit a natural family of weighted summation formulae of the shape
\begin{equation}
	\sideset{}{_{\pi \in \mathcal{F}}} \sum w(\pi) a(\pi) = \text{``Geometric side''}.
\label{SumF}
\end{equation}
	If we are interested in a single term, say $a(\pi_0)$, we sum the LHS of these formulae with suitable weights, so that the contribution of the term indexed by $\pi_0$ is ``amplified'', in the sense that its contribution to the final formula becomes dominant compared with other components on the LHS. Consequently, a bound of the RHS can be regarded as a good bound of the selected $a(\pi_0)$. In order for this method to work, the ``Geometric side'' is expected to have non-trivial cancellation, easy to detect.
	
	In the first applications of this principle, the underlying summation formulae were, in terms of modern language of automorphic representation theory, some relative trace formulae. Then the weights are given by a choice of test function $f: \gp{G} \to \ag{C}$ if the relevant group is $\gp{G}$. It has succeeded in many different situations, such as bounding Fourier coefficients or central $L$-values, the later known as subconvexity problem, for automorphic forms for $\GL_2$ over $\ag{Q}$. For example, for the subconvexity problem for $\GL_2 \times \GL_1$, amplifications of the Petersson-Kuznetsov formulae culminate in \cite{BH08, BH12_Add}.
	
	However, further development with relative trace formulae seems to be technically difficult. The generalization to the number field case presents non trivial computational problems. In the case of subconvexity problem, the generalization to higher degree $L$-functions, but within the group $\GL_2$, does not seem to guarantee even the convex bound uniformly \cite{IM01}, while the relative trace formulae for higher rank group seems to be currently not fine enough for reasonably good analytic number theoretic results.
	
	With this background, Venkatesh \cite{Ve10} and Michel \& Venkatesh \cite{MV10} give a further innovation to the method of amplification, where the underlying summation formulae in (\ref{SumF}) are replaced by the Plancherel formula in different context. As we mentioned above, in order for (\ref{SumF}) to work, non-trivial cancellation on the ``Geometric side'' should be easily detected. Unlike the case for relative trace formulae, where this kind of cancellation is guaranteed by the bounds of (sums of) Kloosterman sums hence from algebraic geometry, the Michel \& Venkatesh method exploits the equidistributions of sub-manifolds. In the context of this paper, the explanation of the cancellation on the new ``Geometric side'' was the main concern of the beginning part of \cite[\S 3]{Wu14}. We refer to the original \cite{Ve10} for more other possible situations where the method can apply. We also remark that, in our previous presentation \cite[\S 3]{Wu14}, it suffices to replace the first step towards bounding the global period, i.e., the Cauchy-Schwarz inequality, with an equality, to recover the underlying summation formulae (\ref{SumF}), as well as the amplifier we used thereafter.
	
	In our previous work \cite{Wu14}, we have made part of \cite{MV10}, i.e., the subconvex bound of $L(1/2,\pi \otimes \chi)$, where $\pi$ is a \emph{fixed} cuspidal representation of $\GL_2$ over an arbitrary number field $\F$, and $\chi$ is a \emph{varying} Hecke character, explicit in terms of the analytic conductor $\Cond(\chi)$ of $\chi$. We have not made that bound explicit in terms of the analytic conductor $\Cond(\pi)$ of $\pi$. For its applications to problems like the harmonic analytic approach to Linnik's equidistribution problem on the $3$-dimensional sphere or some related variants \cite{CU05}, it is important to know at least that the dependence on $\Cond(\pi)$ is polynomial. More importantly, in our recent attempt to make the subconvex component explicit in the work of \cite{MV10} for $\GL_2$, the exponent of $\Cond(\pi_{\fin})$ in the bound of $L(1/2,\pi \otimes \chi)$ enters directly into the final subconvex saving. These constitute the main motivation of the current paper. The main bound will be given in Theorem \ref{MainThm} with a precise form. For the moment, we content ourselves with the following consequence, which looks more compact.
	
\begin{corollary}
	Let $\Cond(\pi_{\fin})^{\flat}$ be the product of $\Nr(\vp)$ over all prime ideals $\vp$ at which $\pi$ is ramified. There is an aboslute constant $C > 0$ such that for any $\epsilon > 0$
	$$ L(1/2, \pi \otimes \chi) \ll_{\F,\epsilon} (\Cond(\pi) \Cond(\chi))^{\epsilon} \Cond(\pi_{\infty})^C \Cond(\pi_{\fin})^{\frac{7}{6}} (\Cond(\pi_{\fin})^{\flat})^{\frac{1}{12} + \frac{\theta}{3}} \Cond(\chi)^{\frac{1}{2}- \frac{1}{8}(1-2\theta)}. $$
\label{SimpMainBd}
\end{corollary}

	\subsection{Local Test Functions at Infinite Places}
	
	If we compare Michel \& Venkatesh's method with the traditional amplification method using relative trace formulae, we easily find that the choice of the test function $\varphi_0 \in \pi$ plays the role of the test function $f$ mentioned in the previous subsection. In the case of relative trace formulae, two different choices of test functions $f$ can lead to two different results which do not necessarily cover one over the other \cite{Iw84, LS95}. Hence it is reasonable to ask about similar possibilities for different choices of $\varphi_0$. Looking into the technical details of \cite{Wu14}, it is not hard to guess one important reason for which we have chosen $\varphi_0$ at archimedean places corresponding to fixed bump functions on $\ag{R}^{\times}$ or $\ag{C}^{\times}$ in the Kirillov model: it is mainly for the technical convenience of bounding local terms. This choice has a formal non-consistence with the choice at finite places, where the new vectors have been specified. Then what happens if we choose the new vectors also at the archimedean places? In general, i.e., if $\pi$ is allowed to vary with varying central character, the new-vector-version does not hopefully seem to give a better result than what the original one does. However, in some special cases, for example when the central character of $\pi$ is fixed at the infinite places, i.e., under the \emph{Assumption (A)} below in \S \ref{MR}, we shall see that the new-vector-version, i.e., \emph{Option (B)} in \S \ref{ArchChoice}, works equally well. This will have some technical convenience for situations like \cite{Wu2}, because it implies that after Cauchy-Schwarz only $\gp{K}$-finite or even $\gp{K}$-invariant vectors appear in the spectral decomposition. In this sense, Option (B) is somewhat a better choice than the original one in \cite{Wu2}, where the subconvexity for $L$-functions associated with Hecke characters is treated.
	
	In general, this paper reveals all the impacts of the choice of local test functions at infinite places. Locally, the choice of $\varphi_0$ should make Proposition \ref{MainLocBd} below work. This is already clear in \cite{Wu14}. Globally, we should be able to effectively bound the $\intL^4$-norm of $X.\varphi_0$ for $X$ in the universal developing algebra of the Lie algebra of $\GL_2(\ag{A}_{\infty})$ in terms of the analytic conductor $\Cond(\pi)$ of $\pi$, see Proposition \ref{L4CrudeBd}. However, neither choice sounds to yield the optimal result. We hope that the presentation given in this paper makes the criteria of good test functions clearer than the hitherto existing papers in the literature.
	
\begin{remark}
	For the new-vector-version of local choice at infinite places, our treatment is not complete when there are complex places (see \emph{Assumption (B)} below in \S \ref{MR}). This seems to be only a technical issue and should be removable if finer analysis were available, but the computation would be much too heavy. Hence we decide not to carry out the computation in other cases. However, even with these restrictions, what we treat is still sufficient for applications in situations like \cite{Wu2}. Moreover, the local test vectors we choose are in fact \emph{minimal vectors} in the sense of Definition \ref{MinVec}. The analogue and convenience of such test vectors at finite places has been exploited in \cite{HNS18}.
\end{remark}
\begin{remark}
	We have not made efforts to compare the effects on the final outcomes as exponents of $\Cond(\pi_{\infty})$ of both choices under \emph{Assumptions (A) \& (B)}, only because it is unimportant for the applications we have in mind.
\end{remark}
\begin{remark}
	The main technical tools for the analysis with new vectors at archimedean places are two lemmas, i.e., Lemma \ref{FourErd} \& \ref{BesselErd}, which seem to be new in the theory of asymptotic analysis and of independent interest.
\end{remark}

	\subsection{$\intL^4$-norm of the Test Function}
	
	The $\intL^4$-norm of (some derivatives of) the test function $\varphi_0$ appear in the final bound. We have used a period method to bound the relevant $\intL^4$-norm of the test function $\varphi_0$ in Corollary \ref{L4CrudeBd}. However, we call it a \emph{crude bound} due to the following reasons:
\begin{itemize}
	\item[(1)] Our bound of the local factors at ramified places is via bounding the absolute value of the relevant matrix coefficients, which does not seem to give the true size by comparison with an existing computation in some special cases due to Nelson, Pitale and Saha (see Remark \ref{ReasonCrude}). Moreover, a strange term $\Cond(\pi_{\fin})^{\flat}$ comes into the bound because of the sharp bound of $\Cond(\mathrm{Ad} \pi_{\fin})$ in terms of $\Cond(\pi_{\fin})$ \cite[Proposition 2.5]{NPS13}.
	\item[(2)] As an alternative approach, one may apply the sup-norm techniques, such as those in \cite{BHMM16, Sah17, As17}, to bound the relevant $\intL^4$-norm. The main difficulty is that all these existing sup-norm bounds are only available for functions which are spherical at infinite places. But once the sup-norm bound for the test function considered in this paper is available, the $\intL^4$-norm bound should no longer involve $\Cond(\pi_{\fin})^{\flat}$ and go beyond what the method of period can offer. Precisely, we expect that the sup-norm bound can improve the bound in Proposition \ref{L4CrudeBd} to
	$$ \Cond(\pi_{\infty})^N \Cond(\pi_{\fin})^{\frac{1}{2}+\epsilon}, $$
which should improve the main bound in Theorem \ref{MainThm} to
\begin{align*}
	&(\Cond(\pi_{\fin}) \Cond(\chi))^{\epsilon} \Cond(\pi_{\infty})^C \Cond(\chi)^{\frac{1}{2}} \cdot \\
	&\max \left\{ \Cond(\pi_{\fin})^{\frac{1}{4}} \Cond_{\fin}[\pi,\chi]^{\frac{1}{8}} \Cond(\chi)^{- \frac{1}{8}(1-2\theta)}, \Cond(\pi_{\fin})^{\frac{3}{4}} \Cond_{\fin}(\pi,\chi)^{\frac{\theta}{3}} \Cond(\chi)^{-\frac{1}{6}} \right\},
\end{align*}
	or the simplified main bound in Corollary \ref{SimpMainBd} to
	$$ (\Cond(\pi_{\fin}) \Cond(\chi))^{\epsilon} \Cond(\pi_{\infty})^C \Cond(\pi_{\fin})^{\frac{3}{4}} (\Cond(\pi_{\fin})^{\flat})^{\frac{\theta}{3}} \Cond(\chi)^{\frac{1}{2}- \frac{1}{8}(1-2\theta)}. $$
\end{itemize}

	\subsection{Organization of the Paper}
	
	This paper is a refinement of our previous work \cite{Wu14}. Our experience tells us that there are a lot of transitions between global and local computations. The global computations reveal the structure of the proof, and are often reduced to the local computations, which are technical in nature. It is also worthwhile to compare the local computations in different parts, in order to better understand how our choice of test vectors make the proof work. Hence, instead of a linear exhibition according to the logical order, we decide to organize the current paper more like lecture notes of a talk.
	
	Precisely, in \S 2 the proof of a lemma or proposition including the main result Theorem \ref{MainThm} there, which is of global nature, should be regarded as a synthesis of the relevant local computations. These proofs will be always given \emph{before} the relevant local results are available and \emph{motivate} the investigation of the relevant local questions. The sentences contained in such proofs also serve as \emph{pointers} either directly to the relevant local technical ingredients in \S 3, or to the technical computations of global nature in \S 4. In this way, we find the presentation of the current paper is more compact and more elegant than our previous \cite{Wu14}. We strongly encourage the reader to read each such proof \emph{twice}: first from global to local, then from local to global.
	
	We give technical details of the local computations in \S 3. In \S 3.2, special cares about our investigation on some possible variants of the local test vectors at infinite places are taken.
	
	In \S 4, we give proofs of some technical global computations. Note that the proofs contain further \emph{pointers} to the relevant local computations in \S 3. We have tried our best to optimize the estimations given before \S 4.
	
	The relevant $\intL^4$-norm of the test function is given in \S 5. The estimation is not optimized. This part is of a flavor quite different from the main body of text. We intend to develop and optimize this part later in a future paper, when more technical results are available.
	
	Some very technical computations, including some seemly new results in the asymptotic analysis which are crucial for our variant of local test functions at infinite places to work, are given in the \S 6 Appendix. This part, together with the technical computations concerning our variant of choice of local archimedean test functions in \S 3.2, can be skipped for the first reading.

\section{Sketch and First Reductions}

	\subsection{Notations and Main Result}
	\label{MR}

	We first give a list of basic notations relevant to this paper:
\begin{itemize}
	\item $\F$: base number field with absolute discriminant $\Dis(\F)$ and degree $d_{\F}=[\F:\ag{Q}]$;
	\item $\pi$: varying cuspidal representation of $\GL_2$ with central character $\omega$;
	\item $\chi$: varying Hecke character;
	\item $\Cond_{\fin}(\pi)^{\flat} := \prod q_{\vp}, q_{\vp} := \Nr(\vp)$ where $\vp$ runs over primes such that $\cond(\pi_{\vp}) > 0$;
	\item $\Cond_{\fin}(\pi, \chi) := \prod q_{\vp}, q_{\vp} := \Nr(\vp)$ where $\vp$ runs over primes such that $\cond(\pi_{\vp}), \cond(\chi_{\vp}) > 0$;
	\item $\Cond_{\fin}[\pi, \chi] := \prod \Cond(\pi_{\vp})$ where $\vp$ runs over primes such that $\cond(\pi_{\vp}), \cond(\chi_{\vp}) > 0$;
	\item $\varphi_0 \in \pi$: unitary function/vector specified in the Kirillov model, ``new'' at $\vp < \infty$, two options (Option (A) \& (B)) at $v \mid \infty$ given in \S \ref{ArchChoice};
	\item $\varphi = n(T).\varphi_0 \in \pi$: $T \in \ag{A}$, given in Lemma \ref{LocLBdFinite}, \ref{LocLBdReal} and \ref{LocLBdCp}, $\Norm[T] := \sideset{}{_{v: T_v \neq 0}} \prod \norm[T_v]_v$.
\end{itemize}
	
	Note that Option (A) for the test function $\varphi_0$ corresponds to the original choice of \cite{MV10, Wu14}. Option (B) is a variant that we shall investigate under the following \emph{restrictions}:
\begin{itemize}
	\item \emph{Assumption (A):} $\omega_{\infty}$ is a fixed character;
	\item \emph{Assumption (B):} if $v$ is any complex place and $\pi_v = \pi(\mu_1,\mu_2)$ for two (quasi-)characters $\mu_1,\mu_2$ of $\ag{C}^{\times}$, then $\mu=\mu_1 \mu_2^{-1}$ is such that $\mu(\rho e^{i\alpha})=\rho^{i\tau}$ for some $\tau \in \ag{R}$, where $\rho > 0, 0 \leq \alpha < 2\pi$.
\end{itemize}
	These restrictions are not essential, but removing them demands too much technical work. As we doubt on the optimality of both Option (A) and Option (B) for the final bound, we do not elaborate on removing them in this paper. The results obtained for Option (B) are sufficient for applications to \cite{Wu2}.
	
	For other notations, we import those in \cite[\S 2.1]{Wu14}, with the following differences or emphasis:
\begin{itemize}

	\item[(1)] The number field is written in bold character $\F$, with ring of algebraic integers $\vo$ and ring of adeles $\ag{A}$. $v$ denotes a place of $\F$. If $v < \infty$ is finite, we usually write $v=\vp$, which is identified with a prime ideal $\vp$ of $\vo$.

	\item[(2)] We write the algebraic groups defined over $\F$ in bold characters such as $\grG, \gp{N}, \grB,\grZ$ etc, where $\grG = \GL_2$, $\grB$ is the upper triangular subgroup of $\grG$, $\gp{N} \vartriangleleft \gp{B}$ is the unipotent upper triangular subgroup, and $\grZ$ is the center of $\grG$.

	\item[(3)] $\gp{K} = \sideset{}{_v} \prod \gp{K}_v$ is the standard maximal compact subgroup of $\GL_2(\ag{A})$, i.e.
	$$ \gp{K}_v = \left\{ \begin{matrix} \SO_2(\ag{R}) & \text{if } \F_v = \ag{R} \\ \SU_2(\ag{C}) & \text{if } \F_v = \ag{C} \\ \GL_2(\vo_{\vp}) & \text{if } v=\vp < \infty \end{matrix} \right. . $$
	
	\item[(4)] In $\GL_2$, for local or global variables $x \in \F_v$ or $\ag{A}$, $y \in \F_v^{\times}$ or $\ag{A}^{\times}$, we write
	$$ n(x) = \begin{pmatrix} 1 & x \\ & 1 \end{pmatrix}, \quad a(y) = \begin{pmatrix} y & \\ & 1 \end{pmatrix}. $$
	
	\item[(5)] We use the abbreviation
	$$ [\GL_2] = \GL_2(\F) \gp{Z}(\ag{A}) \backslash \GL_2(\ag{A}) = [\PGL_2]. $$
	
	\item[(6)] If $f_0 \in \pi(1,1)$ is in the global principal series representation induced from trivial characters, which defines a flat section $f_s \in \pi(\norm_{\ag{A}}^s, \norm_{\ag{A}}^{-s})$, we normalize the usual Eisenstein series $\eis(s,f_0) = \eis(f_s)$ by
	$$ \eis^*(s,f_0) := \Lambda_{\F}(1+2s) \eis(s,f_0). $$
	
	\item[(7)] In the above equation, $\Lambda_{\F}(s)$ is the complete Dedekind zeta function of $\zeta_{\F}(s)$. More generally, $L(\cdot)$ denotes $L$-functions without factors at infinity. $\Lambda(\cdot)$ denotes complete $L$-functions.
\end{itemize}

	The main result of this paper is as follows.

\begin{theorem}
	There is an aboslute constant $C > 0$ such that for any $\epsilon > 0$
\begin{align*}
	L(1/2, \pi \otimes \chi) &\ll_{\F,\epsilon} (\Cond(\pi_{\fin}) \Cond(\chi))^{\epsilon} \Cond(\pi_{\infty})^C \Cond(\chi)^{\frac{1}{2}} \cdot \\
	&\quad \max\left\{ \Cond(\pi_{\fin})^{\frac{3}{4}} (\Cond(\pi_{\fin})^{\flat})^{\frac{1}{16}} \Cond_{\fin}[\pi,\chi]^{\frac{1}{8}} \Cond(\chi)^{- \frac{1}{8}(1-2\theta)}, \right. \\
	&\quad \left. \Cond(\pi_{\fin})^{\frac{7}{6}} (\Cond(\pi_{\fin})^{\flat})^{\frac{1}{12}} \Cond_{\fin}(\pi,\chi)^{\frac{\theta}{3}} \Cond(\chi)^{-\frac{1}{6}} \right\},
\end{align*}
	where the dependence on $\F$ is polynomial in $\Dis(\F)$, exponential in the degree $[\F:\ag{Q}]$.
\label{MainThm}
\end{theorem}
\begin{proof}
	By (\ref{IntRepsL}), we are reduced to bounding from above
	$$ \sideset{}{_v} \prod \ell_v(s, W_{\varphi,v}, \chi_v)^{-1} \quad \text{and} \quad \int_{\F^{\times} \backslash \ag{A}^{\times}} \varphi(a(y)) \chi(y) d^{\times}y $$
	for our choice of test function $\varphi$ (option (A)). The product of local terms is bounded as $\ll_{\F, \epsilon} \Cond(\pi)^{\epsilon} \Cond(\chi)^{1/2}$ according to Proposition \ref{MainLocBd}. Recollecting (\ref{TruncBd}), Lemma \ref{CstBd}, Lemma \ref{CuspTypBd} into (\ref{CuspTermDef}), Lemma \ref{EisTypUBd} into (\ref{EisTermUDef}), Lemma \ref{EisTypRBd} into (\ref{EisTermRDef}), together with Proposition \ref{L4CrudeBd}, the period is bounded as (we omit the polynomial dependence on $\Cond(\pi_{\infty})$)
\begin{align*}
	(\Cond(\pi)\Cond(\chi))^{\epsilon} &\cdot \max(\Cond_{\fin}(\pi, \chi)^{\theta} \Cond(\pi_{\fin})^{\frac{1}{2}} \Cond(\chi)^{-\frac{\kappa}{2}}, E^{-1}, \\
	&\quad \Cond(\pi_{\fin})^{\frac{3}{2}} (\Cond(\pi_{\fin})^{\flat})^{\frac{1}{8}} \Cond_{\fin}[\pi,\chi]^{\frac{1}{4}} E \Cond(\chi)^{-\frac{1-2\theta}{4}}, \Cond(\pi_{\fin})^{\frac{3}{2}} (\Cond(\pi_{\fin})^{\flat})^{\frac{1}{8}} \Cond(\chi)^{\frac{\kappa-1}{4}}).
\end{align*}
	We conclude upon choosing $E = \Cond(\pi_{\fin})^{-\frac{3}{4}} (\Cond(\pi_{\fin})^{\flat})^{-\frac{1}{16}} \Cond_{\fin}[\pi,\chi]^{-\frac{1}{8}} \Cond(\chi)^{\frac{1-2\theta}{8}}$ and $\kappa$ which equalizes
	$$ \Cond_{\fin}(\pi, \chi)^{\theta} \Cond(\pi_{\fin})^{\frac{1}{2}} \Cond(\chi)^{-\frac{\kappa}{2}} = \Cond(\pi_{\fin})^{\frac{3}{2}} (\Cond(\pi_{\fin})^{\flat})^{\frac{1}{8}} \Cond(\chi)^{\frac{\kappa-1}{4}}. $$
\end{proof}

\noindent Bounding $\Cond_{\fin}[\pi,\chi]$ and $\Cond_{\fin}(\pi,\chi)$ in terms of $\Cond(\pi_{\fin})$ and $\Cond(\chi_{\fin})$ in the worst case, we obtain easily Corollary \ref{SimpMainBd}.

	\subsection{Setup}

	We normalize the local norms on the Whittaker functions so that
	$$ \Norm[\varphi]_{[\PGL_2]}^2 = 2 L(1, \pi, {\rm Ad}) (\sideset{}{_{v \mid \infty}} \prod \zeta_v(2) \zeta_v(1)^{-1})  \sideset{}{_v} \prod \Norm[W_{\varphi,v}]^2. $$
	Precisely, the local norms are defined by (c.f. \cite[Lemma 2.10]{Wu14})
\begin{align*}
	& \Norm[W_{\varphi,v}]^2 = \int_{\F_v^{\times}} \norm[W_{\varphi,v}(a(y))]^2 d^{\times}y, \quad v \mid \infty; \\
	& \Norm[W_{\varphi,\vp}]^2 = \zeta_{\vp}(2) L(1, \pi_{\vp} \times \bar{\pi}_{\vp})^{-1} \int_{\F_{\vp}^{\times}} \norm[W_{\varphi,\vp}(a(y))]^2 d^{\times}y, \quad \vp < \infty.
\end{align*}
	Similarly, for $\tau \in \ag{R}$, any Hecke character $\xi$ and $\Phi \in \Ind_{\gp{B}(\ag{A}) \cap \gp{K}}^{\gp{K}}(\xi,\xi^{-1})^{\infty}$, we normalize the local norms on the Whittaker functions so that
	$$ \left( \int_{\gp{K}} \norm[\Phi(\kappa)]^2 d\kappa \right)^{1/2} =: \Norm[\eis(i\tau,\Phi)]_{{\rm Eis}} = \sideset{}{_v} \prod \Norm[W(i\tau,\Phi)]_v. $$
	Precisely, the local norms are defined by (c.f. \cite[Lemma 2.8]{Wu14})
	$$ \Norm[W(i\tau,\Phi)]_v^2 = \frac{\zeta_v(2)}{\zeta_v(1)^2} \int_{\F_v^{\times}} \extnorm{W(i\tau,\Phi)(a(y))}^2 d^{\times}y. $$

	Recall that we have defined for $E > 0$
\begin{equation}
	S(E) := \{ \vp: E \leq q_{\vp} \leq 2E; \F_{\vp}, \chi_{\vp}, \pi_{\vp} \text{ are unramified} \}, \quad \sigma := \norm[S(E)]^{-2} \sideset{}{_{\vp_1,\vp_2 \in S(E)}} \sum \delta_{q_{\vp_1} q_{\vp_2}^{-1}},
\label{RegMeas}
\end{equation}
	where $\sigma$ is regarded as a measure on $\ag{R}_+$.

	For any $\varphi \in \pi^{\infty}$, recall the Hecke-Jacquet-Langlands' integral representation of $L$-function
\begin{equation}
	\int_{\F^{\times} \backslash \ag{A}^{\times}} \varphi(a(y)) \chi(y) \norm[y]_{\ag{A}}^{s-1/2} d^{\times}y =: \zeta(s,\varphi,\chi) = L(s,\pi \otimes \chi) \prod_v \ell_v(s, W_{\varphi,v}, \chi_v)
\label{IntRepsL}
\end{equation}
	where the local factors are defined by
\begin{align*}
	& \ell_v(s, W_{\varphi,v}, \chi_v) = \int_{\F_v^{\times}} W_{\varphi,v}(a(y)) \chi_v(y) \norm[y]_v^{s-1/2} d^{\times}y, \quad v \mid \infty, \\
	& \ell_{\vp}(s, W_{\varphi,\vp}, \chi_{\vp}) = \frac{\int_{\F_{\vp}^{\times}} W_{\varphi,\vp}(a(y)) \chi_{\vp}(y) \norm[y]_{\vp}^{s-1/2} d^{\times}y}{L_{\vp}(s, \pi_{\vp} \otimes \chi_{\vp})}, \quad \vp < \infty,
\end{align*}
	so that for all but finitely many place $\vp$, $\ell_{\vp}(s, \cdots)=1$. We also write $\ell_v(W_{\varphi,v}, \chi_v)$ for $\ell_v(1/2, W_{\varphi,v}, \chi_v)$, $\zeta(s, \varphi)$ resp. $\ell_v(s,W_{\varphi,v})$ for $\zeta(s, \varphi, 1)$ resp. $\ell_v(s, W_{\varphi,v}, 1)$.
	
	We extend the above integral representation to the case of Eisenstein series. For $\tau \in \ag{R}$, any Hecke character $\xi$ and $\Phi \in \Ind_{\gp{B}(\ag{A}) \cap \gp{K}}^{\gp{K}}(\xi,\xi^{-1})^{\infty}$, define
\begin{align}
	&\quad \int_{\F^{\times} \backslash \ag{A}^{\times}} (\eis(i\tau,\Phi) - \eisCst(i\tau,\Phi))(a(y)) \norm[y]_{\ag{A}}^{s-1/2} d^{\times}y =: \zeta(s, \eis(i\tau,\Phi)) \label{IntRepsLEis} \\
	&= \frac{L(s+i\tau,\xi) L(s-i\tau,\xi^{-1})}{L(1+2i\tau,\xi^2)} \prod_v \ell_v(s, W(i\tau,\Phi_v)) \nonumber
\end{align}
	where the local factors are defined by
\begin{align*}
	& \ell_v(s, W(i\tau,\Phi_v)) = \int_{\F_v^{\times}} W(i\tau,\Phi_v)(a(y)) \norm[y]_v^{s-1/2} d^{\times}y, \quad v \mid \infty, \\
	& \ell_{\vp}(s, W(i\tau,\Phi_{\vp}), \chi_{\vp}) = L_{\vp}(1+2i\tau,\xi_{\vp}^2) \frac{\int_{\F_{\vp}^{\times}} W(i\tau,\Phi_v)(a(y)) \norm[y]_{\vp}^{s-1/2} d^{\times}y}{L_{\vp}(s+i\tau, \xi_{\vp}) L_{\vp}(s-i\tau, \xi_{\vp}^{-1})}, \quad \vp < \infty,
\end{align*}
	so that for all but finitely many place $\vp$, $\ell_{\vp}(s, \cdots)=1$. $\zeta(s, \eis(i\tau,\Phi))$ is holomorphic unless $\xi$ is trivial on $\ag{A}^{(1)}$. If $\xi=1$ and $\tau \neq 0$, $\zeta(s, \eis(i\tau,\Phi))$ admits two simple poles at $s = 1 \pm i\tau$ with residue
\begin{align}
	\zeta^*(1+i\tau, \eis(i\tau,\Phi)) &= \zeta_{\F}^*(1) \prod_v \ell_v(1+i\tau, W(i\tau,\Phi_v)) \label{IntRepsLEisRes1} \\
	\zeta^*(1-i\tau, \eis(i\tau, \Phi)) &= \zeta_{\F}^*(1) \frac{\zeta_{\F}(1-2i\tau)}{\zeta_{\F}(1+2i\tau)} \prod_v \ell_v(1-i\tau, W(i\tau,\Phi_v)). \label{IntRepsLEisRes2}
\end{align}

	\subsection{Proofs of Main Bounds}
	\label{MainBds}
	
	We depart from (\ref{IntRepsL}) with our chosen $\varphi$.
\begin{proposition}
	For Option (A) in the general case resp. Option (B) given in \S \ref{ArchChoice} under \emph{Assumptions (A) \& (B)}, there is a constant $C \geq 0$ such that
	$$ \extnorm{ \prod_v \ell_v(W_{\varphi,v}, \chi_v) } \gg_{d_{\F}} L(1, \pi, \mathrm{Ad})^{-1} \Cond(\chi)^{-1/2} \quad \text{resp.} \quad L(1, \pi, \mathrm{Ad})^{-1} \Cond(\pi_{\infty})^{-C} \Cond(\chi)^{-1/2}. $$
\label{MainLocBd}
\end{proposition}
\begin{proof}
	This follows from Lemma \ref{LocLBdFinite}, \ref{LocLBdReal}, \ref{LocLBdCp} and our normalization of local norms.
\end{proof}

	Lemma \ref{TruncEst} easily implies
\begin{align}
	&\quad \extnorm{ \zeta(1/2, \varphi, \chi) - \int_{\F^{\times} \backslash \ag{A}^{\times}} (\sigma*h)(\norm[y]_{\ag{A}}) \varphi(a(y)) \chi(y) d^{\times}y } \nonumber \\
	&\ll_{\F, \epsilon} (\Cond(\pi) \Cond(\chi))^{\epsilon} \Cond(\pi_{\infty})^C \Cond_{\fin}(\pi, \chi)^{\theta} \Cond(\pi_{\fin})^{1/2} \Cond(\chi)^{-\kappa/2}, \label{TruncBd}
\end{align}
	where $h$ is the same choice as in Lemma \ref{TruncEst} and $\sigma$ is defined in (\ref{RegMeas}). Note that
	$$ \int_{\F^{\times} \backslash \ag{A}^{\times}} (\sigma*h)(\norm[y]_{\ag{A}}) \varphi(a(y)) \chi(y) d^{\times}y = \int_{\F^{\times} \backslash \ag{A}^{\times}} h(\norm[y]_{\ag{A}}) (\sigma_{\chi}*\varphi)(y) \chi(y) d^{\times}y $$
deduced from a change of variables, where $\sigma_{\chi}$ is the adjoint measure on $\ag{A}^{\times}$ of $\sigma$ defined by
	$$ \sigma_{\chi} := \norm[S(E)]^{-2} \sideset{}{_{\vp_1,\vp_2 \in S(E)}} \sum \chi_{\vp_1}(\varpi_{\vp_1}) \chi_{\vp_2}(\varpi_{\vp_2})^{-1} \delta_{a(\varpi_{\vp_1})a(\varpi_{\vp_2}^{-1})}. $$
	In order to ease and unify the notations, we replace $\ell^{\chi,h}$ resp. $\ell^h$ in \cite[Lemma 3.2]{Wu14} by
	$$ \ell^{\chi,h}(\cdot) =: \zeta(h, \cdot, \chi), \quad \text{resp.} \quad \ell^h(\cdot) =: \zeta(h, \cdot). $$
	For $\vec{\vp}=(\vp_1, \vp_2, \vp_1', \vp_2') \in S(E)^4$, we also write
	$$ \chi_{\vec{\vp}} := \chi_{\vp_1}(\varpi_{\vp_1}) \chi_{\vp_2}(\varpi_{\vp_2})^{-1} \overline{\chi_{\vp_1'}(\varpi_{\vp_1'}) \chi_{\vp_2'}(\varpi_{\vp_2'})^{-1}}. $$
	By Cauchy-Schwarz and opening the square, we get
\begin{align*}
	\extnorm{\zeta(h, \sigma_{\chi}*\varphi, \chi)}^2 &= \extnorm{\int_{\F^{\times} \backslash \ag{A}^{\times}} h(\norm[y]_{\ag{A}}) (\sigma_{\chi}*\varphi)(y) \chi(y) d^{\times}y}^2 \nonumber \\
	&\ll \log (\Cond(\pi_{\fin}) \Cond(\chi)) \cdot \left\{ \norm[S(E)]^{-4} \sum_{\vec{\vp} \in S(E)^4} \chi_{\vec{\vp}} \zeta\left(h, \left( a\left( \frac{\varpi_{\vp_1}}{\varpi_{\vp_2}} \right).\varphi_0 \overline{a\left( \frac{\varpi_{\vp_1'}}{\varpi_{\vp_2'}} \right).\varphi_0} \right)_{\gp{N}} \right) + \right. \nonumber \\
	&\quad \norm[S(E)]^{-4} \sum_{\vec{\vp} \in S(E)^4} \chi_{\vec{\vp}} \zeta\left(h, n(T).\left( a\left( \frac{\varpi_{\vp_1}}{\varpi_{\vp_2}} \right).\varphi_0 \overline{a\left( \frac{\varpi_{\vp_1'}}{\varpi_{\vp_2'}} \right).\varphi_0} \right)_{{\rm cusp}} \right) + \nonumber \\
	&\quad \left. \norm[S(E)]^{-4} \sum_{\vec{\vp} \in S(E)^4} \chi_{\vec{\vp}} \zeta\left(h, n(T).\left( a\left( \frac{\varpi_{\vp_1}}{\varpi_{\vp_2}} \right).\varphi_0 \overline{a\left( \frac{\varpi_{\vp_1'}}{\varpi_{\vp_2'}} \right).\varphi_0} \right)_{{\rm Eis}} \right) \right\},
\end{align*}
where $f = f_{\gp{N}} + f_{{\rm cusp}} + f_{{\rm Eis}}$ is the Fourier inversion decomposition in the sense of \cite[Theorem 2.18]{Wu14}. In what follows, we only discuss the typical situation for $\vec{\vp}$, in the sense of \emph{Type 1} in \cite[Proposition 3.5]{Wu14}, i.e., $\vp_1, \vp_2, \vp_1', \vp_2'$ are distinct. The other cases contribute at most the same.

	We have by Mellin inversion
\begin{align}
	&\quad \zeta\left(h, \left( a\left( \frac{\varpi_{\vp_1}}{\varpi_{\vp_2}} \right).\varphi_0 \overline{a\left( \frac{\varpi_{\vp_1'}}{\varpi_{\vp_2'}} \right).\varphi_0} \right)_{\gp{N}} \right) \nonumber \\
	&= \int_{\Re s = \epsilon} \Mellin{h}(-s) \zeta\left(s+1/2, \left( a\left( \frac{\varpi_{\vp_1}}{\varpi_{\vp_2}} \right).\varphi_0 \overline{a\left( \frac{\varpi_{\vp_1'}}{\varpi_{\vp_2'}} \right).\varphi_0} \right)_{\gp{N}} \right) \frac{ds}{2\pi i}. \label{CstTermDef}
\end{align}
\begin{lemma}
	The total contribution from the constant part is
\begin{align*}
	&\quad \extnorm{ \norm[S(E)]^{-4} \sum_{\vec{\vp} \in S(E)^4} \chi_{\vec{\vp}} \zeta\left(h, \left( a\left( \frac{\varpi_{\vp_1}}{\varpi_{\vp_2}} \right).\varphi_0 \overline{a\left( \frac{\varpi_{\vp_1'}}{\varpi_{\vp_2'}} \right).\varphi_0} \right)_{\gp{N}} \right) } \\
	&\leq \norm[S(E)]^{-4} \sum_{\vec{\vp} \in S(E)^4} \extnorm{ \zeta\left(h, \left( a\left( \frac{\varpi_{\vp_1}}{\varpi_{\vp_2}} \right).\varphi_0 \overline{a\left( \frac{\varpi_{\vp_1'}}{\varpi_{\vp_2'}} \right).\varphi_0} \right)_{\gp{N}} \right) } \ll_{\F, \epsilon} (\Cond(\pi) \Cond(\chi) E)^{\epsilon} E^{-2}.
\end{align*}
\label{CstBd}
\end{lemma}
\begin{proof}
	This is a refinement of \cite[Lemma 3.4]{Wu14}. We get it by Inserting (\ref{MellinhBd}), Lemma \ref{CstEstRef} \& \ref{PNTRSRef} and the prime number theorem $\norm[S(E)] \gg_{\F} E/\log E$.
\end{proof}

	For $e \in \Bas(\pi')$, where $\pi' \subset \intL^2([\PGL_2])$ is a cuspidal representation such that for $\vp < \infty$
\begin{equation}
	\cond(\pi_{\vp}') \leq \left\{ \begin{matrix} \cond(\pi_{\vp}) & \text{if } \vp \notin \vec{\vp} \\ 1 & \text{if } \vp \in \vec{\vp}, \end{matrix} \right.
\label{CondPi'}
\end{equation}
	and where $\Bas(\pi')$ is the orthonormal basis given in the table right after \cite[Remark 6.4]{Wu14}, writing the Fourier coefficient as
	$$ C_{\vec{\vp}}(\varphi_0,e) := \extPairing{a\left( \frac{\varpi_{\vp_1}}{\varpi_{\vp_2}} \right).\varphi_0 \overline{a\left( \frac{\varpi_{\vp_1'}}{\varpi_{\vp_2'}} \right).\varphi_0}}{e} = \int_{[\PGL_2]} a\left( \frac{\varpi_{\vp_1}}{\varpi_{\vp_2}} \right).\varphi_0(g) \overline{a\left( \frac{\varpi_{\vp_1'}}{\varpi_{\vp_2'}} \right).\varphi_0(g)} \overline{e(g)} dg, $$
	we have by Mellin inversion
\begin{align}
	&\quad \zeta\left(h, n(T).\left( a\left( \frac{\varpi_{\vp_1}}{\varpi_{\vp_2}} \right).\varphi_0 \overline{a\left( \frac{\varpi_{\vp_1'}}{\varpi_{\vp_2'}} \right).\varphi_0} \right)_{{\rm cusp}} \right) \nonumber \\
	&= \int_{\Re s = 0} \Mellin{h}(-s) \cdot \sideset{}{_{\pi'}} \sum \sideset{}{_{e \in \Bas(\pi')}} \sum C_{\vec{\vp}}(\varphi_0,e) \zeta(s+1/2, n(T).e) \frac{ds}{2\pi i}. \label{CuspTermDef}
\end{align}
\begin{lemma}
	There exist a set $\mathcal{D}^2$ of pairs of differential operators from $\SL_2(\F_{\infty})$ of absolutely finite cardinality and absolutely finite degree, and an absolute constant $B$ such that
\begin{align*}
	&\quad \sideset{}{_{\pi'}} \sum \sideset{}{_{e \in \Bas(\pi')}} \sum \extnorm{ C_{\vec{\vp}}(\varphi_0,e) \zeta(s+1/2, n(T).e) } \\
	&\ll_{\F, \epsilon} (1+\norm[s])^{B/2} \left( \sum_{(X_1,X_2) \in \mathcal{D}^2} \Norm[X_1.\varphi_0]_4 \Norm[X_2.\varphi_0]_4 \right) (\Cond(\pi_{\fin})E^4)^{1/2+\epsilon} \Cond_{\fin}[\pi,\chi]^{1/2} \Norm[T]^{-(1/2-\theta)+\epsilon}.
\end{align*}
\label{CuspTypBd}
\end{lemma}
\begin{proof}
	This is a refinement of \cite[(6.16)]{Wu14}. Inserting Lemma \ref{CuspEstRef} and \ref{LindelAvg}, we bound the LHS as
\begin{align*}
	&\quad \sideset{}{_{\pi'}} \sum \sideset{}{_{e \in \Bas(\pi')}} \sum \extnorm{C_{\vec{\vp}}(\varphi_0,e)} (\dim \gp{K}_{\infty}.e_{\infty})^{1/2} \extnorm{ \frac{L(s+1/2, \pi')}{\sqrt{ L(1,\pi', \mathrm{Ad}) }} } \cdot \Norm[T]^{-(1/2-\theta)+\epsilon} \Cond_{\fin}[\pi,\chi]^{1/2} \\
	&\leq \left( \sideset{}{_{\pi'}} \sum \sideset{}{_{e \in \Bas(\pi')}} \sum \extnorm{C_{\vec{\vp}}(\varphi_0,e)}^2 (\dim \gp{K}_{\infty}.e_{\infty}) \lambda_{e,\infty}^A \right)^{1/2}  \cdot \\
	&\quad \left( \sideset{}{_{\pi'}} \sum \sideset{}{_{e \in \Bas(\pi')}} \sum \extnorm{ \frac{L(s+1/2, \pi')}{\sqrt{ L(1,\pi', \mathrm{Ad}) }}}^2 \lambda_{e,\infty}^{-A} \right)^{1/2} \cdot \Norm[T]^{-(1/2-\theta)+\epsilon} \Cond_{\fin}[\pi,\chi]^{1/2} \\
	&\ll_{\F,\epsilon} \extNorm{ \Delta_{\infty}^{A/2} (-\Casimir_{\gp{K}_{\infty}})^{1/2} \left( a\left( \frac{\varpi_{\vp_1}}{\varpi_{\vp_2}} \right).\varphi_0 \overline{a\left( \frac{\varpi_{\vp_1'}}{\varpi_{\vp_2'}} \right).\varphi_0} \right) }_2 \cdot \\
	&\quad (1+\norm[s])^{B/2} (\Cond(\pi_{\fin}) E^4)^{1/2+\epsilon} \Cond_{\fin}[\pi,\chi]^{1/2} \Norm[T]^{-(1/2-\theta)+\epsilon}.
\end{align*}
	We finally write the differential operator $\Delta_{\infty}^{A/2} (-\Casimir_{\gp{K}_{\infty}})^{1/2}$ as linear combination of products of degree $1$ differential operators, possibly with Sobolev interpolation, and deduce the existence of $\mathcal{D}^2$ such that the first factor in the last line is bounded as
\begin{align*}
	&\quad \sum_{(X_1,X_2) \in \mathcal{D}^2} \extNorm{ a\left( \frac{\varpi_{\vp_1}}{\varpi_{\vp_2}} \right).X_1.\varphi_0 \overline{a\left( \frac{\varpi_{\vp_1'}}{\varpi_{\vp_2'}} \right).X_2.\varphi_0}  }_2 \\
	&\leq \sum_{(X_1,X_2) \in \mathcal{D}^2} \extNorm{ a\left( \frac{\varpi_{\vp_1}}{\varpi_{\vp_2}} \right).X_1.\varphi_0 }_4 \extNorm{a\left( \frac{\varpi_{\vp_1'}}{\varpi_{\vp_2'}} \right).X_2.\varphi_0}_4 = \sum_{(X_1,X_2) \in \mathcal{D}^2} \extNorm{ X_1.\varphi_0 }_4 \extNorm{X_2.\varphi_0}_4.
\end{align*}
\end{proof}

	Let $\xi$ run over characters of $\F^{\times} \backslash \ag{A}^{(1)}$, extended to a Hecke character by triviality on $\ag{R}_+$ according to a fixed section $s_{\F}: \ag{R}_+ \to \ag{A}^{\times}$, such that for $\vp < \infty$
\begin{equation}
	\cond(\xi_{\vp}) \leq \left\{ \begin{matrix} \cond(\pi_{\vp})/2 & \text{if } \vp \notin \vec{\vp} \\ 0 & \text{if } \vp \in \vec{\vp}. \end{matrix} \right.
\label{CondXi}
\end{equation}
	Denote by $\Bas(\xi)$ the orthonormal basis of $\Ind_{\gp{B}(\ag{A}) \cap \gp{K}}^{\gp{K}} (\xi, \xi^{-1})$, selected according to the same principle of the table right after \cite[Remark 6.4]{Wu14}. For $\Phi \in \Bas(\xi), \tau \in \ag{R}$, write the Fourier coefficient as
\begin{align*}
	C_{\vec{\vp}}(\varphi_0,\Phi; i\tau) &:= \extPairing{a\left( \frac{\varpi_{\vp_1}}{\varpi_{\vp_2}} \right).\varphi_0 \overline{a\left( \frac{\varpi_{\vp_1'}}{\varpi_{\vp_2'}} \right).\varphi_0}}{\eis(i\tau,\Phi)} \\
	&= \int_{[\PGL_2]} a\left( \frac{\varpi_{\vp_1}}{\varpi_{\vp_2}} \right).\varphi_0(g) \overline{a\left( \frac{\varpi_{\vp_1'}}{\varpi_{\vp_2'}} \right).\varphi_0(g)} \overline{\eis(i\tau,\Phi)(g)} dg.
\end{align*}
	We have by Mellin inversion
\begin{align}
	&\quad \zeta\left(h, n(T).\left( a\left( \frac{\varpi_{\vp_1}}{\varpi_{\vp_2}} \right).\varphi_0 \overline{a\left( \frac{\varpi_{\vp_1'}}{\varpi_{\vp_2'}} \right).\varphi_0} \right)_{{\rm Eis}} \right) \nonumber \\
	&= \int_{\Re s \gg 1} \Mellin{h}(-s) \cdot \sideset{}{_{\xi}} \sum \sideset{}{_{\Phi \in \Bas(\xi)}} \sum \int_{\ag{R}} C_{\vec{\vp}}(\varphi_0,\Phi; i\tau) \zeta(s+1/2, n(T).\eis(i\tau, \Phi)) \frac{d\tau}{4\pi} \frac{ds}{2\pi i}. \nonumber \\
	&= \int_{\Re s = 0} \Mellin{h}(-s) \cdot \sideset{}{_{\xi}} \sum \sideset{}{_{\Phi \in \Bas(\xi)}} \sum \int_{\ag{R}} C_{\vec{\vp}}(\varphi_0,\Phi; i\tau) \zeta(s+1/2, n(T).\eis(i\tau, \Phi)) \frac{d\tau}{4\pi} \frac{ds}{2\pi i} \label{EisTermUDef} \\
	&\quad \begin{matrix} + \sideset{}{_{\Phi \in \Bas(1)}} \sum \int_{\ag{R}} C_{\vec{\vp}}(\varphi_0,\Phi; i\tau) \Mellin{h}(-(1/2+i\tau)) \zeta^*(1+i\tau, n(T).\eis(i\tau, \Phi)) \frac{d\tau}{4\pi} \\ + \sideset{}{_{\Phi \in \Bas(1)}} \sum \int_{\ag{R}} C_{\vec{\vp}}(\varphi_0,\Phi; i\tau) \Mellin{h}(-(1/2-i\tau)) \zeta^*(1-i\tau, n(T).\eis(i\tau, \Phi)) \frac{d\tau}{4\pi}. \end{matrix} \label{EisTermRDef}
\end{align}
\begin{lemma}
	There exists a set $\mathcal{D}^2$ of pairs of differential operators from $\SL_2(\F_{\infty})$ of absolutely finite cardinality and absolutely finite degree and an absolute constant $B$ such that
\begin{align*}
	&\quad \sideset{}{_{\xi}} \sum \sideset{}{_{\Phi \in \Bas(\xi)}} \sum \int_{\ag{R}} \extnorm{ C_{\vec{\vp}}(\varphi_0,\Phi; i\tau) \zeta(s+1/2, n(T).e) } \frac{d\tau}{4\pi} \\
	&\ll_{\F, \epsilon} (1+\norm[s])^{B/2} \left( \sum_{(X_1,X_2) \in \mathcal{D}^2} \Norm[X_1.\varphi_0]_4 \Norm[X_2.\varphi_0]_4 \right) \Cond(\pi_{\fin})^{1/2+\epsilon} \Cond_{\fin}[\pi,\chi]^{1/2} \Norm[T]^{-1/2+\epsilon}.
\end{align*}
\label{EisTypUBd}
\end{lemma}
\begin{proof}
	Inserting Lemma \ref{EisEstURef} and \ref{LindelAvgEis}, we bound the LHS as
\begin{align*}
	&\quad \sideset{}{_{\xi}} \sum \sideset{}{_{\Phi \in \Bas(\xi)}} \sum \int_{\ag{R}} \extnorm{ C_{\vec{\vp}}(\varphi_0,\Phi; i\tau) } (\dim \gp{K}_{\infty}.\Phi_{\infty})^{1/2} \extnorm{ \frac{L(1/2+s+i\tau,\xi) L(1/2+s-i\tau,\xi^{-1})}{L(1+2i\tau,\xi^2)} } \frac{d\tau}{4\pi} \\
	&\quad \cdot \Norm[T]^{-1/2+\epsilon} \Cond_{\fin}[\pi,\chi]^{1/2} \\
	&\leq \left( \sideset{}{_{\xi}} \sum \sideset{}{_{\Phi \in \Bas(\xi)}} \sum \int_{\ag{R}} \extnorm{ C_{\vec{\vp}}(\varphi_0,\Phi; i\tau) }^2 (\dim \gp{K}_{\infty}.\Phi_{\infty}) \lambda_{\Phi,i\tau,\infty}^A \frac{d\tau}{4\pi} \right)^{1/2} \cdot \Norm[T]^{-1/2+\epsilon} \Cond_{\fin}[\pi,\chi]^{1/2} \cdot \\
	&\quad \left( \sideset{}{_{\xi}} \sum \sideset{}{_{\Phi \in \Bas(\xi)}} \sum \int_{\ag{R}} \extnorm{ \frac{L(1/2+s+i\tau,\xi) L(1/2+s-i\tau,\xi^{-1})}{L(1+2i\tau,\xi^2)} }^2 \lambda_{\Phi,i\tau,\infty}^{-A} \frac{d\tau}{4\pi} \right)^{1/2} \\
	&\ll_{\F,\epsilon} \extNorm{ \Delta_{\infty}^{A/2} (-\Casimir_{\gp{K}_{\infty}})^{1/2} \left( a\left( \frac{\varpi_{\vp_1}}{\varpi_{\vp_2}} \right).\varphi_0 \overline{a\left( \frac{\varpi_{\vp_1'}}{\varpi_{\vp_2'}} \right).\varphi_0} \right) }_2 \cdot \\
	&\quad (1+\norm[s])^{B/2} \Cond(\pi_{\fin})^{1/2+\epsilon} \Cond_{\fin}[\pi,\chi]^{1/2} \Norm[T]^{-1/2+\epsilon}.
\end{align*}
	The remaining argument is the same as the proof of Lemma \ref{CuspTypBd}.
\end{proof}
\begin{lemma}
	There exists a set $\mathcal{D}_c^2$ of pairs of differential operators from $\gp{K}_{\infty}$ of absolutely finite cardinality and absolutely finite degree such that
\begin{align*}
	&\quad \sideset{}{_{\Phi \in \Bas(1)}} \sum \int_{\ag{R}} \extnorm{ C_{\vec{\vp}}(\varphi_0,\Phi; i\tau) \Mellin{h}(-(1/2 \pm i\tau)) \zeta^*(1 \pm i\tau, n(T).\eis(i\tau, \Phi)) } \frac{d\tau}{4\pi} \\
	&\ll_{\F, \epsilon} \left( \sum_{(X_1,X_2) \in \mathcal{D}_c^2} \Norm[X_1.\varphi_0]_4 \Norm[X_2.\varphi_0]_4 \right) \Cond(\pi_{\fin})^{1/2} \Cond(\chi)^{(\kappa+1)/2} \Norm[T]^{-1+\epsilon}.
\end{align*}
\label{EisTypRBd}
\end{lemma}
\begin{proof}
	Inserting (\ref{MellinhBd}) and Lemma \ref{EisEstRRef}, we bound the LHS as
\begin{align*}
	&\quad \sideset{}{_{\Phi \in \Bas(1)}} \sum \int_{\ag{R}} \extnorm{ C_{\vec{\vp}}(\varphi_0,\Phi; i\tau) } \frac{\lambda_{\Phi,i\tau,\infty}^C}{(1+\norm[\tau])^{2C+1}} \frac{d\tau}{4\pi} \cdot \Cond(\chi)^{(\kappa+1)/2} \Norm[T]^{-1} \\
	&\ll \left( \sideset{}{_{\Phi \in \Bas(1)}} \sum \int_{\ag{R}} \extnorm{ C_{\vec{\vp}}(\varphi_0,\Phi; i\tau) }^2 \lambda_{\Phi,\gp{K}_{\infty}}^{2(C+1)} \frac{d\tau}{4\pi} \right)^{1/2} \cdot \left( \sideset{}{_{\Phi \in \Bas(1)}} \sum \int_{\ag{R}} (1+\norm[\tau])^{-2} \lambda_{\Phi,\gp{K}_{\infty}}^{-2} \frac{d\tau}{4\pi} \right)^{1/2} \\
	&\quad \cdot \Cond(\chi)^{(\kappa+1)/2} \Norm[T]^{-1},
\end{align*}
	where we have written $\lambda_{\Phi, \gp{K}_{\infty}}$ the Laplacian eigenvalue of $\Phi$ for $-2 \Casimir_{\gp{K}_{\infty}}$. Noting that the middle term gives $\Cond(\pi_{\fin})^{1/2}$ by dimension estimation, we conclude.
\end{proof}

\section{Local Choices and Estimations}

	We drop the subcript $v$ for simplicity of notations.

	\subsection{Non Archimedean Places}
	
		\subsubsection{Choices and Main Bounds}
		
	Let $W_{\varphi_0,v} = W_0$ be a new vector in the Kirillov model of $\pi$.
\begin{lemma}
	Let $r=\cond(\chi_{\vp})$ resp. $d=\cond(\psi_{\vp})$ be the logarithmic conductor of $\chi_{\vp}$ resp. $\psi_{\vp}$. Choose $T_{\vp} = \varpi_{\vp}^{-(r+d)}$ if $r > 0$ resp. $T_{\vp}=0$ if $r=0$. Then we have
	$$ \extnorm{ \ell_{\vp}(n(T_{\vp}).W_0, \chi_{\vp}) } \geq \Norm[W_0] \left\{ \begin{matrix} \Cond(\chi_{\vp})^{-1/2} & \text{if } r>0; \\ 1 & \text{if } r = 0. \end{matrix} \right. $$
\label{LocLBdFinite}
\end{lemma}
\begin{proof}
	With $\Norm[W_0]$ replaced by $\norm[W_0(1)]$, this is precisely \cite[Corollary 4.8]{Wu14}, or essentially \cite[Lemma 11.7]{Ve10}. The disappearance of the factor $\zeta_{\vp}(1)$ is due to the estimation
	$$ \norm[L_{\vp}(1/2, \pi_{\vp} \otimes \chi_{\vp})]^{-1} \geq \left\{ \begin{matrix} (1-q_{\vp}^{-(1/2+\theta)})(1-q_{\vp}^{-(1/2-\theta)}) & \text{if } \pi_{\vp} \text{ is not square-integrable} \\ 1-q_{\vp}^{-1} & \text{if } \pi_{\vp} \otimes \chi_{\vp} \text{ is Steinberg} \\ 1 & \text{if } \pi_{\vp} \text{ is supercuspidal.} \end{matrix} \right. $$
	We have $\norm[W_0(1)] = \Norm[W_0]$ due to \cite[(2.10)]{Wu14} and our normalization of local norms.
\end{proof}
\begin{remark}
	There seems to be two different notions of ``new vector'' in the litterature, i.e., the one defined by \cite[(4.18)]{Ge75} is not the same as \cite[(4.4)]{JPS81}. \cite[Corollary 4.8]{Wu14} confused the two. We now stick to \cite[(4.4)]{JPS81} as in the statement of \cite[Corollary 4.8]{Wu14}. Also, the second case of \cite[Corollary 4.8]{Wu14} needs to be rectified as above.
\end{remark}
\begin{lemma}
	Let $s \in \ag{C}$ with $\Re s = \sigma > 0$. Then we have
	$$ \extnorm{ \ell_{\vp}(s, n(T_{\vp}).W_0, \chi_{\vp}) } \leq \Norm[W_0] \left\{ \begin{matrix} 8\Cond(\chi_{\vp})^{-1/2} q_{\vp}^{\max(0, \theta - \sigma)} & \text{if } \cond(\chi_{\vp}), \cond(\pi_{\vp})>0; \\ 1 & \text{otherwise.} \end{matrix} \right. $$
\label{LocUBdFinite}
\end{lemma}
\begin{proof}
	If $\pi = \pi(\mu_1,\mu_2)$ and exactly one of $\mu_1,\mu_2$ is unramified, we call it \emph{semi-unramified}. If $\cond(\chi_{\vp}) > 0$, we easily verify the estimation
	$$ \norm[L_{\vp}(s, \pi_{\vp} \otimes \chi_{\vp})]^{-1} \leq \left\{ \begin{matrix} (1+q_{\vp}^{-(\sigma+\theta)})(1+q_{\vp}^{-(\sigma-\theta)}) & \text{if } \pi_{\vp} \otimes \chi_{\vp} \text{ is unramified} \\ 1+q_{\vp}^{-(\sigma-\theta)} & \text{if } \pi_{\vp} \otimes \chi_{\vp} \text{ is semi-unramified} \\ 1+q_{\vp}^{-(1/2+\sigma-\theta)} & \text{if } \pi_{\vp} \otimes \chi_{\vp} \text{ is Steinberg} \\ 1 & \text{otherwise} \end{matrix} \right. $$
Since each of the first three cases implies $\cond(\pi_{\vp}) > 0$ and can be bounded by $4 q_{\vp}^{\max(0, \theta - \sigma)}$ while the bound $\zeta_{\vp}(1) \leq 2$ gives the extra factor $2$, we conclude the proof together with \cite[Corollary 4.8]{Wu14} or \cite[Lemma 11.7]{Ve10}.
\end{proof}

		\subsubsection{Refined Upper Bounds}
		
	We restrict our attention to $\pi=\pi(\norm^{i\tau}, \norm^{-i\tau})$. Let $e_n, n \in \ag{N}$ be an orthonormal basis of the space of ``classical vectors'' \cite[Definition 5.4]{Wu14}.
\begin{lemma}
	We have a relation
	$$ e_1 = c_1^{-1/2} \cdot \left\{ a(\varpi^{-1}).e_0 - q^{-1/2}(1+q^{-1})^{-1}(q^{i\tau} + q^{-i\tau}) e_0 \right\}, $$
	$$ e_n = c^{-1/2} \left\{ a(\varpi^{-n}).e_0 - q^{-1/2}(q^{i\tau} + q^{-i\tau}) a(\varpi^{-(n-1)}).e_0 + q^{-1} a(\varpi^{-(n-2)}).e_0 \right\}, \forall n \geq 2, $$
	$$ \text{with }c_1=1-q^{-1}(1+q^{-1})^{-2}\norm[q^{i\tau} + q^{-i\tau}]^2 \asymp 1, \quad c = 1-q^{-2} - \frac{q^{-1}-q^{-2}-q^{-3}}{(1+q^{-1})^2}\norm[q^{i\tau} + q^{-i\tau}]^2 \asymp 1, $$
	where the asymptotic is taken with respect to $q = \Nr(\vp) \to \infty$.
\end{lemma}
\begin{proof}
	Using the MacDonald's formula \cite[Proposition 4.6.6]{Bu98}, we easily deduce that
	$$ e_n' = a(\varpi^{-n}).e_0 - q^{-1/2}(q^{i\tau} + q^{-i\tau})a(\varpi^{-{n-1}}).e_0 + q^{-1} a(\varpi^{-(n-2)}).e_0, n \geq 2 $$
is orthogonal to $a(\varpi^{-k}).e_0$ for $0 \leq k \leq n-2$, since $\Pairing{a(\varpi^{-n}).e_0}{e_0}$ is of the form $C_1q^{ik\tau} q^{-k/2} + C_2q^{-ik\tau} q^{-k/2}$ with $C_1,C_2$ constants. We verify by direct computation that it is also orthogonal to $a(\varpi^{-{n-1}}).e_0$.
\end{proof}
\begin{lemma}
	Let $W_n$ be the Kirillov function of $e_n$. For $l \in \ag{N}$ and $\Re s = 1$, we have
	$$ \extnorm{ \zeta(s,n(\varpi^{-l}).W_n) } \ll_{\epsilon} q^{n/2} q^{-\max(n,l)(1-\epsilon)}. $$
\label{ExUpBdNArch}
\end{lemma}
\begin{proof}
	This is a refinement of \cite[(4.11)]{Wu14}. We may assume $W_0(1)=1$ and ignore $c,c_1$ in the previous lemma since the normalizations differ by a factor asymptotically equal to $1$. From
	$$ \zeta(s,W_0)= \sum_{k=0}^{\infty} W_0(\varpi^k) q^{-k(s-1/2)} = (1-q^{-(s+i\tau)})^{-1}(1-q^{-(sii\tau)})^{-1} $$
	and the previous lemma, we deduce that
	$$ W_n(\varpi^{n+k}) = q^{-k/2} \sum_{a+b=k} q^{i\tau(a-b)} + (q^{i\tau}+q^{-i\tau}) q^{-(k+2)/2} \sum_{a+b=k+1} q^{i\tau(a-b)} + q^{-(k+4)/2} \sum_{a+b=k+2} q^{i\tau(a-b)}. $$
	It follows that
	$$ \extnorm{W_n(\varpi^{n+k})} \leq (k+1)q^{-k/2} 1_{k \geq 0} + 2(k+2) q^{-(k+2)/2} 1_{k \geq -1} + (k+3) q^{-(k+4)/2} 1_{k \geq -2}. $$
	We conclude by inserting the above bound into the formula, deduced from \cite[Lemma 4.7]{Wu14}
	$$ \extnorm{\zeta(s,n(\varpi^{-l}).W_n)} \leq \sum_{k=l}^{\infty} \extnorm{W_n(\varpi^k)} q^{-k/2} + \frac{1}{q-1} \extnorm{W_n(\varpi^{l-1})} q^{-(l-1)/2}. $$
\end{proof}

	\subsection{Archimedean Places}
	
		\subsubsection{Some Properties of the Kirillov Model}
		
	We proceed under \emph{Assumptions (A) \& (B)}.
\begin{lemma}
	Let $W_0$ be a unitary minimal vector in the Kirillov model of $\pi$. If $\F = \ag{R}$ resp. $\ag{C}$, there is $y_0 \in \F^{\times}$ with $\norm[y_0]_v \asymp \Cond(\pi)^{1/2}$ resp. $\Cond(\pi)$ such that as $\Cond(\pi) \to \infty$
	$$ \norm[W_0(y_0)] \gg \Cond(\pi)^{1/12}, \quad \norm[y_0 W_0'(y_0)] \ll \Cond(\pi)^{7/12} \quad \text{resp.} \quad \norm[W_0(y_0)] \gg \Cond(\pi)^{1/3}, \quad \norm[y_0 W_0'(y_0)] \ll \Cond(\pi)^{5/6}. $$
\label{KirMaxLowerBd}
\end{lemma}
\begin{proof}
	We do not need to consider the case $\pi$ is in complementary series, since $\Cond(\pi) \to \infty$ excludes this case. We shall distinguish:
\begin{itemize}
	\item[(1)] $\F=\ag{R}$, $\pi$ is principal series.
	\item[(2)] $\F=\ag{R}$, $\pi$ is discrete series.
	\item[(3)] $\F=\ag{C}$, $\pi$ is principal series.
\end{itemize}
(1) By twisting, we may assume $\omega=1$ or $\sgn$ and either $\pi = \pi(\norm^{i\tau/2}, \norm^{-i\tau/2})$ or $\pi = \pi(\norm^{i\tau/2}\sgn, \norm^{-i\tau/2})$ for some $\tau \in \ag{R}$. In the first case,
	$$ W_0(y) = \norm[y]^{\frac{1-i\tau}{2}} \int_{\ag{R}} \frac{e^{-2\pi i xy}}{(1+x^2)^{\frac{1}{2}+i\tau}} dx= -(\frac{1}{2}+i\tau) \frac{\norm[y]^{\frac{1-i\tau}{2}}}{2\pi i y} \int_{\ag{R}} \frac{2x}{(1+x^2)^{\frac{3}{2}+i\tau}} e^{-2\pi i xy} dx, $$
	$$ yW_0'(y) = \frac{(1+i\tau)(1+2i\tau)}{8\pi i \norm[y]^{\frac{1+i\tau}{2}}} \int_{\ag{R}} \frac{2x e^{-2\pi i xy}}{(1+x^2)^{\frac{3}{2}+i\tau}} dx + (\frac{1}{2}+i\tau) \norm[y]^{\frac{1-i\tau}{2}} \int_{\ag{R}} \frac{2x^2 e^{-2\pi i xy}}{(1+x^2)^{\frac{3}{2}+i\tau}} dx. $$
	We then have
\begin{equation} 
	\extnorm{ W_0(\frac{\tau}{2\pi}) } \asymp \norm[\tau]^{1/2} \extnorm{ \int_{\ag{R}} \frac{2x}{(1+x^2)^{\frac{3}{2}}} e^{-i\tau (x+\log (1+x^2)} dx} \asymp \norm[\tau]^{1/6},
\label{NVKirMaxRealEst}
\end{equation}
\begin{equation} 
	\extnorm{ \frac{\tau}{2\pi} W_0'(\frac{\tau}{2\pi}) } \asymp \norm[\tau]^{7/6},
\label{NVKirDMaxRealEst}
\end{equation}
where we applied Lemma \ref{Erd} with $x_0=-1, m=3$. In the second case
\begin{align*}
	W_0(y) &= \norm[y]^{\frac{1-i\tau}{2}} \int_{\ag{R}} \frac{x+i}{(1+x^2)^{1+i\tau}} e^{-2\pi i xy} dx \\
	&= \frac{\norm[y]^{\frac{1-i\tau}{2}}}{2\pi i y} \int_{\ag{R}} \frac{1}{(1+x^2)^{1+i\tau}} e^{-2\pi i xy} dx - (1+i\tau) \frac{\norm[y]^{\frac{1-i\tau}{2}}}{2\pi i y} \int_{\ag{R}} \frac{2x(x+i)}{(1+x^2)^{2+i\tau}} e^{-2\pi i xy} dx,
\end{align*}
	hence (\ref{NVKirMaxRealEst}) \& (\ref{NVKirDMaxRealEst}) remain valid and are proved the same way.
	
\noindent (2) We have $\pi=\pi(\mu_1,\mu_2)$ with $\mu_1 \mu_2^{-1}(t) = t^p \sgn(t)$ for some integer $p > 0$. $W_0$ is computed in \cite[\S 2.13 (80)]{Go70}. With normalization, we get
	$$ W_0(y) = \frac{(4\pi)^{(p+1)/2}}{\Gamma(p+1)^{1/2}} y^{\frac{p+1}{2}} e^{-2\pi y} 1_{y > 0}. $$
	Using Stirling's formula, we see
\begin{equation}
	W_0(\frac{p+1}{4\pi}) \asymp p^{1/4}, \quad \frac{p+1}{4\pi} W_0'(\frac{p+1}{4\pi}) = 0.
\label{NVKirDMaxRealDEst}
\end{equation}

\noindent (3) By twisting we may assume $\pi=\pi(\norm_{\ag{C}}^{i\tau/2}, \norm_{\ag{C}}^{-i\tau/2})$ resp. $\pi = \pi(\alpha^N, \alpha^{-N})$ resp. $\pi = \pi(\alpha^{N+1}, \alpha^{-N})$ for some $\tau > 0$ resp. $N \in \ag{N}$ since we are under \emph{Assumption (B)}. Here $\alpha(\rho e^{i\alpha}) = e^{i\alpha}$. In the first case, by Proposition \ref{WhitNewCp} and Corollary \ref{MinVecCp}, or more directly the formula under \cite[(8.4)]{BHMM16}, we see that this is essentially (1) with $\tau$ replaced by $2\tau$ and with an extra factor $\norm[y]^{1/2}$, i.e.,
\begin{equation}
	\extnorm{W_0(\frac{\tau}{2\pi})} \asymp \norm[\tau]^{2/3}, \quad \extnorm{\frac{\tau}{2\pi} W_0'(\frac{\tau}{2\pi})} \asymp \norm[\tau]^{5/3}.
\label{NVKirMaxCpEst}
\end{equation}
	In the second resp. third case, we have by Proposition \ref{WhitNewCp} and Corollary \ref{MinVecCp}
	$$ W_0(y) = \frac{4y^{N+1} K_0(4\pi y)}{\Gamma_{\ag{C}}(N+1) \sqrt{B(N+1,N+1)}} \quad \text{resp.} \quad W_0(y) = \frac{4y^{N+3/2} K_{1/2}(4\pi y)}{\Gamma_{\ag{C}}(N+3/2) \sqrt{B(N+2,N+1)}}. $$
	Taking into account the asymptotic behavior as $y \to \infty$
	$$ K_0(y) \asymp \sqrt{\frac{\pi}{2y}} e^{-y} = K_1(y), $$
	we get by Stirling's formula
\begin{equation}
	\extnorm{ W_0(\frac{N+1/2}{4\pi}) } \asymp N^{3/4}, \quad \text{resp.} \quad \extnorm{ W_0(\frac{N+1}{4\pi}) } \asymp N^{3/4}.
\label{NVKirMaxCpEstBis}
\end{equation}
\end{proof}
		
		\subsubsection{Choices and Lower Bounds}
		\label{ArchChoice}
	
	We first give the notation of \emph{minimal vectors}, which is crucial for our variant choice of test functions.
\begin{definition}
	Let $\pi_v$ be a unitary irreducible representation of $\GL_2(\F_v)$. For varying character $\chi$ of $\F^{\times}$, there exists (not necessarily unique, see Corollary \ref{MinVecCp} for example) $\chi_0$ such that the (analytic) conductor
	$$ \Cond(\pi_v \otimes \chi_0) = \min_{\chi} \Cond(\pi_v \otimes \chi). $$
	A vector $v_0 \in \pi_v$ is called \emph{minimal} if $v_0 \otimes \chi_0$ is a new vector of $\pi_v \otimes \chi_0$.
\label{MinVec}
\end{definition}
\begin{remark}
	If $v \mid \infty$, it is equivalent to demanding
	$$ \Cond(\omega \chi_0^2) = \min_{\chi} \Cond(\omega\chi^2) $$
	in the above condition. Hence under \emph{Assumption (A)}, $\chi_0$ is fixed.
\end{remark}
		
	We have two options: \\
\noindent (A) Let $\phi \in \Sch(\ag{R}_+^{\times}) \subset \Sch(\ag{R}^{\times})$ be a fixed function and $y_0 \in \ag{R}_+^{\times}$ be such that $\norm[\phi(y_0)] = \sideset{}{_{y \in \ag{R}_+^{\times}}} \max \norm[\phi(y)]$. If $\F=\ag{R}$, let $\varphi_{0,v}$ correspond to $\phi$ in the Kirillov model; if $\F=\ag{C}$, we extend $\phi$ to $\ag{C}^{\times}$ by imposing $\phi(ye^{i\alpha}) = \phi(y)$ for any $\alpha \in \ag{R}/(2\pi \ag{Z})$ and let $\varphi_{0,v}$ correspond to the extended $\phi$ in the Kirillov model. \\
\noindent (B) Under the \emph{Assumptions (A) \& (B)}, let $\varphi_{0,v}$ be a unitary minimal vector corresponding to $W_0$ in the Kirillov model.

\begin{lemma}
	Suppose $\F=\ag{R}$, $\chi(t)=\norm[t]^{i\mu} \sgn^m(t)$ for $m \in \{0,1\}, \mu \in \ag{R}$.
\begin{itemize}
	\item[(1)] If $\norm[\mu] \geq C$ for some absolute constant $C$, choose $T_v = \mu/(2\pi y_0)$. For the option (A), we have
	$$ \extnorm{ \int_{\ag{R}^{\times}} \phi(y) e^{-2\pi i T_vy} \chi(y) d^{\times} y } \gg \norm[\mu]^{-1/2}. $$
	\item[(1')] As in (1), if $\norm[\mu] \leq C$, there exists $T_v$ of absolutely bounded size such that uniformly in $\mu$
	$$ \extnorm{ \int_{\ag{R}^{\times}} \phi(y) e^{-2\pi i T_vy} \chi(y) d^{\times} y } \gg 1. $$
	\item[(2)] Suppose $\pi = \pi(\norm^{i\tau/2}, \norm^{-i\tau/2})$ or $\pi(\norm^{i\tau/2} \sgn, \norm^{-i\tau/2})$ for some $0 \neq \tau \in \ag{R}$ upon twisting. If $\norm[\mu] \gg_{\epsilon} (1+\norm[\tau])^{11/3+\epsilon}$, choose $T_v = \mu/\tau$. For the option (B), we have
	$$ \extnorm{ \int_{\ag{R}^{\times}} W_0(y) e^{-2\pi i T_vy} \chi(y) d^{\times}y } \gg (1+\norm[\tau])^{1/6} \norm[\mu]^{-1/2}. $$
	\item[(2')] As in (2), if $\norm[\mu] \ll (1+\norm[\tau])^4$, for any $\epsilon > 0$ there is $\norm[T_v] \asymp_{\epsilon} (1+\max(\norm[\mu],\norm[\tau]))^{1+\epsilon}$ such that
	$$ \extnorm{ \int_{\ag{R}^{\times}} W_0(y) e^{-2\pi i T_vy} \chi(y) d^{\times}y } \gg_{\epsilon} (1+\max(\norm[\mu],\norm[\tau]))^{-(1+\epsilon)/2} (1+\norm[\tau])^{-1/2}. $$
	\item[(3)] Suppose $\pi = \pi(\mu_1,\mu_2)$ with $\mu_1 \mu_2^{-1}(t) = t^p \sgn(t)$ for some integer $p > 0$. If $\norm[\mu] \gg p^3$, choose $T_v = 2\mu/(p+1)$. For the option (B), we have
	$$ \extnorm{ \int_{\ag{R}^{\times}} W_0(y) e^{-2\pi i T_vy} \chi(y) d^{\times}y } \gg p^{1/4} \norm[\mu]^{-1/2}. $$
	\item[(3')] As in (3), if $\norm[\mu] \leq p^3$, there is $T_v \in [-p^4, p^4]$ such that
	$$ \extnorm{ \int_{\ag{R}^{\times}} W_0(y) e^{-2\pi i T_vy} \chi(y) d^{\times}y } \gg p^{-5/2}. $$
\end{itemize}
\label{LocLBdReal}
\end{lemma}
\begin{proof}
	(1) We only need to apply Lemma \ref{Erd} to
	$$ \int_{\ag{R}^{\times}} \phi(y) e^{-i \mu y / y_0} \chi(y) d^{\times} y = \int_0^{\infty} \phi(e^x) e^{i\mu(x-e^x/y_0)} dx + (-1)^m \int_0^{\infty} \phi(-e^x) e^{i\mu(x+e^x/y_0)} dx. $$
\noindent (1') The proof is included in \cite[Remark 4.4]{Wu14}.
	
\noindent (2) First applying Lemma \ref{Erd}, we get
	$$ \extnorm{ \int_{\ag{R}^{\times}} W_0(y) e^{-2\pi i \mu y/\tau} \chi(y) d^{\times}y } \geq \frac{\sqrt{\pi}}{2} \norm[W_0(\frac{\tau}{2\pi})] \norm[\mu]^{-1/2} - \frac{1}{2} \norm[A.W_0(\frac{\tau}{2\pi})] \norm[\mu]^{-1} + O(\sum_{n=0}^2 \Norm[A^n.W_0]_1) \norm[\mu]^{-1}, $$
	$$ \text{where} \quad A = \begin{bmatrix} 1 & 0 \\ 0 & 0 \end{bmatrix} = y \frac{d}{dy} $$
	is the element in the Lie algebra. The implied constant in $O(\cdots)$ is independent of $\tau$, because by defining
	$$ e^{-2\pi i \mu y/\tau} \chi(y) = e^{i\mu S_{\pm}(x)}, \quad S_{\pm}(x) = x \mp 2 \pi \tau^{-1} e^x,\quad S_{\pm}'(x_0) = 0 $$
	we see that $S_{\pm}(x+x_0) - S_{\pm}(x_0) = x - e^x + 1$ is independent of $\tau$. Using (\ref{NVKirMaxRealEst}), (\ref{NVKirDMaxRealEst}), Lemma \ref{SVKirL1Bd} together with the formulas of the action of the Lie algebra given in \cite[\S 2.7.1]{Wu14}, we get and conclude by
	$$ \extnorm{ \int_{\ag{R}^{\times}} W_0(y) e^{-2\pi i \mu y/\tau} \chi(y) d^{\times}y } \gg \norm[\tau]^{1/6} \norm[\mu]^{-1/2} - O_{\epsilon}((1+\norm[\tau])^{2+\epsilon}) \norm[\mu]^{-1}. $$
	
\noindent (2') Let $h$ be a positive smooth function on $\ag{R}$ such that $0 \leq h \leq 1$, $h(y)=1$ for $\norm[y] \leq 1/2$ and $h(y)=0$ for $\norm[y] \geq 1$. We proceed in three steps. We assume in the following argument that $\norm[T] \gg 1+\max(\norm[\mu], \norm[\tau])$ to simplify some bounds.

\noindent \emph{Step 1:} We have by integration by parts
\begin{align*}
	\int_{\ag{R}^{\times}} (1-h)(y) W_0(y) \chi(y) e^{-2\pi i Ty} d^{\times}y &= \frac{1}{(2\pi i T)^N} \int_{\ag{R}^{\times}} \frac{d^N}{dy^N}((1-h)(y)W_0(y) \chi(y)\norm[y]^{-1}) e^{-2\pi i T y} dy.
\end{align*}
	Writing and proving by induction the existence of polynomials $P_{k,N} \in \ag{Z}[X]$ such that
	$$ A := y\frac{d}{dy}, \quad \frac{d^N}{dy^N} = \sum_{k=1}^N P_{k,N}(y^{-1}) A^k, \quad \deg P_{k,N} \leq 2N-k, P_{k,N}(0) = 0, $$
	taking into account the binomial relation
	$$ A^k (W_0(y) y^{i\mu -1}) = \sum_{l=0}^k \binom{k}{l} (i\mu-1)^{k-l} A^l.W_0(y) y^{i\mu-1}, $$
	we find a bound of the second integral as
	$$ \ll_{n,N} \sum_{k=1}^N \sum_{l=0}^k \binom{k}{l} (1+\norm[\mu])^{k-l} \int_{\norm[y] \geq 1/2} \extnorm{ A^l.W_0(y) } \norm[y]^{-1} d^{\times}y \ll \sum_{k=1}^N \sum_{l=0}^k \binom{k}{l} (1+\norm[\mu])^{k-l} \Norm[A^l.W_0]_2. $$
	Together with the formula of the action of $A$ in \cite[\S 2.7.1]{Wu14}, implying $\Norm[A^l.W_0]_2 \ll_l (1+\norm[\tau])^l$, we deduce
\begin{equation}
	\extnorm{ \int_{\ag{R}^{\times}} (1-h)(y) W_0(y) \chi(y) e^{-2\pi i Ty} d^{\times}y } \ll_{h,N} \norm[T]^{-N} (1+\norm[\mu]+\norm[\tau])^N.
\label{Bd1Real}
\end{equation}
	The bound for the integral for $y<-1$ is the same.
	
\noindent \emph{Step 2:} Let $W_{0,M}$ be the sum of the first $M$-terms in the expansion (\ref{NVKir0ExpReal1}) or (\ref{NVKir0ExpReal2}). Uniformly for $\norm[y] \leq 1$, we have by the same expansion
	$$ \extnorm{W_0(y) - W_{0,M}(y)} \ll \norm[y]^{2M+1/2}(1+\norm[\tau])^{-(M+1/2)} $$
	with absolute implied constant. Hence
\begin{equation}
	\extnorm{ \int_{\ag{R}^{\times}} h(y) (W_0(y) - W_{0,M}(y)) \chi(y) e^{-2\pi i Ty} d^{\times}y } \ll (1+\norm[\tau])^{-(M+1/2)}
\label{Bd2Real}
\end{equation}
	with absolute implied constant (even decaying in $M$). Lemma \ref{FourErd} (``moreover'' part) implies for $n \geq 1$
	$$ \extnorm{ \int_{\ag{R}^{\times}} y^{2n}h(y) \chi(y) \norm[y]^{(1 \mp i\tau)/2} e^{-2\pi i yT} d^{\times}y } \leq n! \frac{(1+\norm[\mu \pm \tau/2])^{2n}}{\norm[2\pi T]^{2n+1/2}} + O_{h,n}(1) (\norm[T] - 2 \norm[\mu \mp \tau/2])^{-(2n+1)}, $$
	Hence for any $\delta > 0$ small and $\norm[T] \gg_{M, \delta} 1+\norm[\mu]+\norm[\tau]$
\begin{equation}
	\extnorm{ \int_{\ag{R}^{\times}} h(y) (W_{0,M}(y) - W_{0,0}(y)) \chi(y) e^{-2\pi i Ty} d^{\times}y } \leq \delta \norm[T]^{-1/2}(1+\norm[\tau])^{-1/2}.
\label{Bd3Real}
\end{equation}

\noindent \emph{Step 3:} Applying Lemma \ref{FourErd} again we get
\begin{equation}
	\extnorm{ \int_{\ag{R}^{\times}} h(y) W_{0,0}(y) \chi(y) e^{-2\pi i Ty} dy } \gg \left( \norm[T]^{-1/2} - (\norm[T] - \norm[\mu \pm \tau/2])^{-1} \right) (1+\norm[\tau])^{-1/2}.
\label{Bd4Real}
\end{equation}
	For $\epsilon > 0$ small, we first take $M>2$ (say $M=3$), then take $N$ large such that $1/(N-1/2) < \epsilon$. For $\norm[T] \gg_{h,M,N} (1+\norm[\mu]+\norm[\tau])^{1+1/(2N-1)} (1+\norm[\tau])^{1/(2N-1)}$, we deduce from (\ref{Bd1Real}), (\ref{Bd2Real}), (\ref{Bd3Real}) and (\ref{Bd4Real}) and conclude by
	$$ \extnorm{ \int_{\ag{R}^{\times}} W_0(y) e^{-2\pi i T_vy} \chi(y) d^{\times}y } \gg \norm[T]^{-1/2} (1+\norm[\tau])^{-1/2}. $$

\noindent (3) We have similarly by Lemma \ref{Erd}
	$$ \extnorm{ \int_{\ag{R}^{\times}} W_0(y) e^{-4\pi i \mu y/(p+1)} \chi(y) d^{\times}y } \geq \frac{\sqrt{\pi}}{2} \norm[W_0(\frac{p+1}{4\pi})] \norm[\mu]^{-1/2} - O(\sum_{n=0}^2 \Norm[A^n.W_0]_1) \norm[\mu]^{-1}. $$
	We can explicitly compute and conclude by
	$$ \Norm[A^n.W_0]_1 \asymp (p+1)^{n-1/4}, $$
	together with (\ref{NVKirDMaxRealDEst}).
	
\noindent (3') Let $N=p/2$ if $2 \mid p$ resp. $(p-1)/2$ if $2 \nmid p$. By integration by parts, we have
\begin{align*}
	f(t) := \int_{\ag{R}^{\times}} W_0(y) e^{-2\pi i yt} \chi(y) d^{\times}y &= \frac{(4\pi)^{(p+1)/2}}{\Gamma(p+1)^{1/2}} (2 \pi)^{-(p+1)/2-i\mu} (1+it)^{-N} \cdot \\
	&\quad \left\{ \begin{matrix} \sideset{}{_{k=0}^{N-1}} \prod (k+1/2+i\mu) \int_0^{\infty} y^{1/2+i\mu} e^{-y(1+it)} dy/y & \text{if } 2 \mid p \\ \sideset{}{_{k=0}^{N-1}} \prod (k+1+i\mu) \int_0^{\infty} y^{1+i\mu} e^{-y(1+it)} dy/y & \text{if } 2 \nmid p. \end{matrix} \right.
\end{align*}
	It follows that
	$$ \norm[f(t)] \leq \pi^{1/2} (\norm[\mu]+p)^N \norm[t]^{-N} \Rightarrow \int_{\norm[t] \geq p^4} \norm[f(t)]^2 dt \ll (\norm[\mu]+p)^{2N} p^{-8N+3} \ll p^{-2N+3}. $$
	But by Plancherel formula, we have
	$$ \int_{\ag{R}} \norm[f(t)]^2 dt = \int_{\ag{R}^{\times}} \norm[W_0(y)]^2 \norm[y]^{-1} d^{\times}y = \frac{4\pi}{p}, $$
	hence for $p$ large, we get and conclude by
	$$ \int_{\norm[t] \leq p^4} \norm[f(t)]^2 dt \gg p^{-1} \Rightarrow \max_{\norm[t] \leq p^4} \norm[f(t)] \gg p^{-5/2}. $$
\end{proof}

\begin{remark}
	If $\chi$ is a character of $\ag{C}^{\times}$ with $\chi(\rho e^{i\alpha}) = \rho^{i\mu} e^{im\alpha}$ for some $\mu \in \ag{R}, m \in \ag{Z}$, then its analytic conductor is defined to be
	$$ \Cond(\chi) := (1 + \mu^2 + m^2)/4. $$
	It contains two parts $\mu^2$ and $m^2$ of different nature: analytic resp. arithmetic.
\end{remark}
\begin{definition}
	If we fix a constant $\delta \in (0,1]$, then as $\Cond(\chi) \to \infty$, (at least) one of the following two cases occurs:
\begin{itemize}
	\item[(1)] $\norm[\mu] \geq \delta \norm[m]$. We call it the \emph{$\delta$-analytically dominating} case, or simply \emph{$\delta$-analytic} case.
	\item[(2)] $\norm[m] \geq \delta \norm[\mu]$. We call it the \emph{$\delta$-arithmetically dominating} case, or simply \emph{$\delta$-arithmetic} case.
\end{itemize}
\end{definition}
\begin{lemma}
	Suppose $\F=\ag{C}$, $\chi(\rho e^{i\alpha}) = \rho^{i\mu} e^{im\alpha}$ for some $\mu \in \ag{R}, m \in \ag{Z}$. Let $\varepsilon_0 := m / \mu$ in the $\delta$-analytic, resp. $\mu / m$ in the $\delta$-arithmetic case.
\begin{itemize}
	\item[(1)] If $\Cond(\chi) \geq C$ for some absolute constant $C$, for the option (A), choose any $T_v \in \ag{C}$ such that $\norm[T_v] = \sqrt{1+\varepsilon_0^2} \norm[\mu]/(4\pi y_0)$ resp. $\sqrt{1+\varepsilon_0^2} \norm[m]/(4\pi y_0)$, then we have
	$$ \extnorm{ \int_{\ag{C}^{\times}} \phi(y) e^{-2\pi i (T_vy + \overline{T_vy})} \chi(y) d^{\times} y } \gg \norm[\mu]^{-1} \quad \text{resp.} \quad \norm[m]^{-1}. $$
	\item[(1')] As in (1), if $\Cond(\chi) \leq C$, there exists $T_v$ of absolutely bounded size such that uniformly in $\chi$
	$$ \extnorm{ \int_{\ag{C}^{\times}} \phi(y) e^{-2\pi i (T_vy + \overline{T_vy})} \chi(y) d^{\times} y } \gg 1. $$
	\item[(2)] Suppose $\pi = \pi(\norm^{i\tau/2}, \norm^{-i\tau/2})$ for some $0 < \tau \in \ag{R}$ upon twisting. In the $\delta$-analytic resp. $\delta$-arithmetic case, if $\norm[\mu]$ resp. $\norm[m] \gg_{\epsilon} (1+\norm[\tau])^{10/3+\epsilon}$, choose any $T_v \in \ag{C}$ such that $\norm[T_v] = \sqrt{1+\varepsilon_0^2} \norm[\mu]/(2\tau)$ resp. $\sqrt{1+\varepsilon_0^2} \norm[m]/(2\tau)$. For the option (B), we have
	$$ \extnorm{ \int_{\ag{C}^{\times}} W_0(y) e^{-2\pi i (T_vy + \overline{T_vy})} \chi(y) d^{\times}y } \gg (1+\norm[\tau])^{2/3} \norm[\mu]^{-1} \quad \text{resp.} \quad (1+\norm[\tau])^{2/3} \norm[m]^{-1}. $$
	\item[(2')] As in (2), if $\norm[\mu], \norm[m] \ll (1+\norm[\tau])^4$, for any $\epsilon > 0$ there is $\norm[T] \asymp_{\epsilon} \max((1+\max(\norm[\mu],\norm[\tau], \norm[m]))^{1+\epsilon}, m^2)$ such that
	$$ \extnorm{ \int_{\ag{C}^{\times}} W_0(y) e^{-2\pi i (Ty+\overline{Ty})} \chi(y) d^{\times}y } \gg_{\epsilon} \max((1+\max(\norm[\mu],\norm[\tau], \norm[m]))^{1+\epsilon}, m^2)^{-1} (1+\norm[\tau])^{-1}. $$
%	\item[(3)] Suppose $\pi = \pi(\alpha^N, \alpha^{-N})$ or $\pi(\alpha^{N+1}, \alpha^{-N})$ for some $N \in \ag{N}$ upon twisting. In the $\delta$-analytic resp. $\delta$-arithmetic case, if $\norm[\mu]$ resp. $\norm[m] \gg_{\epsilon} N^{13/4+\epsilon}$, choose any $T_v \in \ag{C}$ such that $\norm[T_v] = \sqrt{1+\varepsilon_0^2} \norm[\mu]/(N+1/2)$ or $\sqrt{1+\varepsilon_0^2} \norm[\mu]/(N+1)$ resp. $\norm[T_v] = \sqrt{1+\varepsilon_0^2} \norm[m]/(N+1/2)$ or $\sqrt{1+\varepsilon_0^2} \norm[m]/(N+1)$. For the option (B), we have
%	$$ \extnorm{ \int_{\ag{C}^{\times}} W_0(y) e^{-2\pi i (T_vy + \overline{T_vy})} \chi(y) d^{\times}y } \gg N^{3/4} \norm[\mu]^{-1} \quad \text{resp.} \quad N^{3/4} \norm[m]^{-1}. $$
\end{itemize}	
%\textcolor{blue}{TODO}
\label{LocLBdCp}
\end{lemma}
\begin{proof}
	(1) We take the $\delta$-analytic case for example, the other being similar. Writing $T_v = \norm[T_v] e^{i\theta}$, we have
	$$ \int_{\ag{C}^{\times}} \phi(y) e^{-2\pi i (T_vy + \overline{T_vy})} \chi(y) d^{\times} y = e^{-im\theta} \int_{-\infty}^{\infty} \int_0^{2\pi} \phi(e^x) e^{i \norm[\mu] \left( \pm x + \varepsilon_0 \alpha - y_0^{-1}\sqrt{1+\varepsilon_0^2} e^x \cos \alpha \right)} d\alpha dx. $$
	The phase function $S_{\pm}(x,\alpha) = \pm x + \varepsilon_0 \alpha - y_0^{-1}\sqrt{1+\varepsilon_0^2} e^x \cos \alpha$ is tempered (Definition \ref{TemPhase} \& Remark \ref{TemPhaseExt}). It has a unique non degenerate critical point (\cite[\S 3.5]{EZ03}) $(x_0,\alpha_0) \in \ag{R} \times \ag{R}/(2\pi \ag{Z})$ satisfying
	$$ e^{x_0} = y_0, \quad \cos \alpha_0 = \pm 1 / \sqrt{1+\varepsilon_0^2}, \quad \sin \alpha_0 = - \varepsilon_0 / \sqrt{1+\varepsilon_0^2}. $$
	We can thus apply Lemma \ref{StatPhase} and conclude by the continuous dependence on $\varepsilon_0 \in [-\delta^{-1}, \delta^{-1}]$ and the compactness of this interval.

\noindent (1') The proof is (again) included in \cite[Remark 4.4]{Wu14}.

\noindent (2) We take the $\delta$-analytic case for example, the other being similar. First applying Lemma \ref{StatPhase}, we get
\begin{align*}
	\extnorm{ \int_{\ag{C}^{\times}} W_0(y) e^{-2\pi i (T_vy + \overline{T_vy})} \chi(y) d^{\times}y } &\geq \frac{\pi}{2} (1+\varepsilon_0^2)^{-1/2} \norm[\mu]^{-1} \extnorm{W_0(\frac{\tau}{2\pi})} \\
	&- \norm[\mu]^{-2} O_{\epsilon}(\sum_{n=0}^2 \Norm[A^n.W_0]_1 + \left( \sum_{n=0}^4 \Norm[A^n.W_0]_2 \right)^{1-\epsilon} \left( \sum_{n=0}^5 \Norm[A^n.W_0]_2 \right)^{\epsilon} ).
\end{align*}
	$$ \text{where} \quad A = \begin{bmatrix} 1 & 0 \\ 0 & 0 \end{bmatrix} = y \frac{d}{dy} $$
	is the element in the Lie algebra. The implied constant in $O_{\epsilon}(\cdots)$ is independent of $\tau$, because by defining
	$$ e^{-2\pi i \sqrt{1+\varepsilon_0^2} (y+\bar{y}) \norm[\mu]/(2\tau)} \chi(y) = e^{i \norm[\mu] S_{\pm}(x, \alpha)}, $$
	$$ S_{\pm}(x, \alpha) = \pm x + \varepsilon_0 \alpha - 2 \pi \tau^{-1} \sqrt{1+\varepsilon_0^2} e^x \cos \alpha, \quad \nabla S_{\pm}(x_0, \alpha_0) = 0, $$
	we see that $\cos \alpha_0 = \pm 1/ \sqrt{1+\varepsilon_0^2}$, $\sin \alpha_0 = - \varepsilon_0 / \sqrt{1+\varepsilon_0^2}$ and
	$$ S_{\pm}(x+x_0, \alpha + \alpha_0) - S_{\pm}(x_0, \alpha_0) = \pm x + \varepsilon_0 \alpha - \sqrt{1+\varepsilon_0^2} (e^x \cos (\alpha + \alpha_0) - \cos \alpha_0) $$
	are independent of $\tau$. Using (\ref{NVKirMaxCpEst}), Lemma \ref{SVKirL1Bd} together with the formulas of the action of the Lie algebra given in \cite[\S 2.7.2]{Wu14}, we get and conclude by
	$$ \extnorm{ \int_{\ag{C}^{\times}} W_0(y) e^{-2\pi i (T_vy + \overline{T_vy})} \chi(y) d^{\times}y } \gg (1+\norm[\tau])^{2/3} \norm[\mu]^{-1} - O_{\epsilon}((1+\norm[\tau])^{4+\epsilon}) \norm[\mu]^{-2}. $$
	
\noindent (2') Let $h$ be a positive smooth function on $\ag{R}_+$ such that $0 \leq h \leq 1$, $h(r)=1$ for $0 \leq r \leq 1/2$ and $h(r)=0$ for $r \geq 1$. Let $h(y)$ be the extension of $h$ to $\ag{C}$ by defining $h(re^{i\alpha})=h(r)$. We proceed in three steps. We assume in the following argument that $\norm[T] \gg 1+\max(\norm[\mu], \norm[m]^2, \norm[\tau])$ to simplify some bounds. We may also assume $T > 0$. Recall the Laplacian $\Delta = \partial^2/\partial x^2 + \partial^2/\partial y^2$ can be written in the spherical coordinates as
	$$ \Delta = \frac{\partial^2}{\partial r^2} + \frac{1}{r} \frac{\partial}{\partial r} + \frac{1}{r^2} \frac{\partial^2}{\partial \alpha^2}, \quad x+iy = r e^{i\alpha}. $$

\noindent \emph{Step 1:} We have by integration by parts
\begin{align*}
	&\quad \int_{\ag{C}^{\times}} (1-h)(y) W_0(y) \chi(y) e^{-2\pi i T (y+\bar{y})} d^{\times}y \\
	&= \frac{1}{(-4\pi T^2)^N} \int_{\ag{R}^{\times}} (\Delta^*)^N ((1-h)(y)W_0(y) \chi(y)\norm[y]^{-1}) e^{-2\pi i T (y+\bar{y})} dy.
\end{align*}
	The dual Laplacian can be written as
	$$ \Delta^* = \frac{1}{r^2}\left( (A-1)^2 + \frac{\partial^2}{\partial \alpha^2} \right), \quad A = r \frac{\partial}{\partial r}. $$
	It follows, by induction, that for any $N \in \ag{N}$ there exist polynomials $P_{k,l,N} \in \ag{Z}[X]$ such that
	$$ (\Delta^*)^N = \sum_{k+2l \leq 2N} P_{k,l,N}(\frac{1}{r^2}) A^k \frac{\partial^{2l}}{\partial \alpha^{2l}}, \quad P_{k,l,N}(0)=0. $$
	Arguing as in the real case, we get
\begin{equation}
	\extnorm{ \int_{\ag{C}^{\times}} (1-h)(y) W_0(y) e^{-2\pi i (T_vy + \overline{T_vy})} \chi(y) d^{\times}y } \ll_N \norm[T]^{-2N} (1+\norm[\mu] + \norm[m] + \norm[\tau])^{2N}.
\label{Bd1Cp}
\end{equation}
	
\noindent \emph{Step 2:} Let $W_{0,M}$ be the sum of the first $M$-terms in the expansion (\ref{NVKir0ExpCp}). Uniformly for $\norm[y] \leq 1$, we have by the same expansion
	$$ \extnorm{W_0(y) - W_{0,M}(y)} \ll \norm[y]^{2M+1/2}(1+\norm[\tau])^{-(M+1/2)} $$
	with absolute implied constant. Hence
\begin{equation}
	\extnorm{ \int_{\ag{C}^{\times}} h(y) (W_0(y) - W_{0,M}(y)) \chi(y) e^{-2\pi i T(y+\bar{y})} d^{\times}y } \ll (1+\norm[\tau])^{-(M+1)}
\label{Bd2Cp}
\end{equation}
	with absolute implied constant (even decaying in $M$). For $1 \leq n < M$, Lemma \ref{BesselErd} implies
	$$ \extnorm{ \int_{\ag{C}^{\times}} h(y) \norm[y]^{2n \pm i\tau} \chi(y) e^{-2\pi i (Ty + \overline{Ty})} dy } \ll_n \frac{(1+\norm[\mu]+\norm[\tau]+\norm[m])^{2n}}{\norm[T]^{2n+1}}. $$
	Hence for any $\delta > 0$ small and $\norm[T] \gg_{M, \delta} 1+\norm[\mu]+\norm[\tau]+\norm[m]$
\begin{equation}
	\extnorm{ \int_{\norm[y] \leq 1} (W_{0,M}(y) - W_{0,0}(y)) \chi(y) e^{-2\pi i (Ty+\overline{Ty})} d^{\times}y } \leq \delta \norm[T]^{-1}(1+\norm[\tau])^{-1}.
\label{Bd3Cp}
\end{equation}

\noindent \emph{Step 3:} Applying Lemma \ref{BesselErd} again we get
\begin{equation}
	\extnorm{ \int_{\ag{C}^{\times}} h(y) W_{0,0}(y) \chi(y) e^{-2\pi i T(y+\bar{y})} dy } \gg \left( \norm[T]^{-1} - \norm[T]^{-1/2} (\norm[T] - \norm[2\mu \pm \tau])^{-1} \right) (1+\norm[\tau])^{-1}.
\label{Bd4Cp}
\end{equation}
	For $\epsilon > 0$ small, we first take $M>2$ (say $M=3$), then take $N$ large such that $1/(N-1/2) < \epsilon$. For $\norm[T] \gg_{h,M,N} \max((1+\norm[\mu]+\norm[\tau]+\norm[m])^{1+1/(2N-1)} (1+\norm[\tau])^{1/(2N-1)}, \norm[m]^2)$, we deduce from (\ref{Bd1Cp}), (\ref{Bd2Cp}), (\ref{Bd3Cp}) and (\ref{Bd4Cp}) and conclude by
	$$ \extnorm{ \int_{\ag{C}^{\times}} W_0(y) e^{-2\pi i T_v(y+\bar{y})} \chi(y) d^{\times}y } \gg \norm[T]^{-1} (1+\norm[\tau])^{-1}. $$
	
\noindent (3) We take the $\delta$-analytic case for example and assume $\pi = \pi(\alpha^N, \alpha^{-N})$ for definiteness. Applying Lemma \ref{StatPhase}, we get
\begin{align*}
	\extnorm{ \int_{\ag{C}^{\times}} W_0(y) e^{-2\pi i (T_vy + \overline{T_vy})} \chi(y) d^{\times}y } &\geq \frac{\pi}{2} (1+\varepsilon_0^2)^{-1/2} \norm[\mu]^{-1} \extnorm{W_0(\frac{N+1/2}{4\pi})} \\
	&- \norm[\mu]^{-2} O_{\epsilon}(\sum_{n=0}^2 \Norm[A^n.W_0]_1 + \left( \sum_{n=0}^4 \Norm[A^n.W_0]_2 \right)^{1-\epsilon} \left( \sum_{n=0}^5 \Norm[A^n.W_0]_2 \right)^{\epsilon} ).
\end{align*}
	We can explicitly compute and estimate
	$$ \Norm[A^n.W_0]_1 \asymp N^{n-3/4}, \quad \Norm[A^n.W_0]_2 \asymp N^n, \quad 0 \leq n \leq 4, $$
	using the following well-known formulas for Bessel-$K$ functions
	$$ K_l(y) > 0, l \in \ag{N}, \quad K_l'(y) = -K_{l+1}(y) + \frac{l}{y} K_l(y), $$
	$$ \int_0^{\infty} y^N K_{\alpha}(y) dy = 2^{N-1} \Gamma((N+1+\alpha)/2) \Gamma((N+1-\alpha)/2), $$
	together with the formulas of the action of the Lie algebra given in \cite[\S 2.7.2]{Wu14}. We deduce and conclude by
	$$ \extnorm{ \int_{\ag{C}^{\times}} W_0(y) e^{-2\pi i (T_vy + \overline{T_vy})} \chi(y) d^{\times}y } \gg N^{3/4} \norm[\mu]^{-1} - N^{4+\epsilon} \norm[\mu]^{-2}. $$
\end{proof}

		\subsubsection{Refined Upper Bounds}
		
\begin{lemma}
	Let $W$ be the Kirillov function of a $\gp{K}$-isotypic vector in $\pi=\pi(\norm_v^{i\tau}, \norm_v^{-i\tau})$, with eigenvalue $\lambda_W$ for the Laplacian $\Delta_v$, local component of $\Delta_{\infty}$ defined in Lemma \ref{LindelAvg}. For some absolute constant $C>0$, we have
	$$ \extnorm{ \int_{\ag{R}^{\times}} W(y) \norm[y]^{1/2+i\tau} e^{-2\pi i Ty} d^{\times}y } \quad \text{resp.} \quad \extnorm{ \int_{\ag{C}^{\times}} W(y) \norm[y]_{\ag{C}}^{1/2+i\tau} e^{-2\pi i (Ty+\overline{Ty})} d^{\times}y } \ll \frac{\lambda_W^C}{\norm[T]_v} \Norm[W]_2. $$
\label{ExUpBdArch}
\end{lemma}
\begin{proof}
	This is the counterpart of Lemma \ref{ExUpBdNArch}, a refinement/correction of the last paragraph of \cite[\S 4.3]{Wu14}. We give detail for the real case, the complex case being similar. Note that we have a trivial bound
	$$ \extnorm{ \int_{\ag{R}^{\times}} W(y) \norm[y]^{1/2+i\tau} e^{-2\pi i Ty} d^{\times}y } \leq \int_{\ag{R}^{\times}} \extnorm{W(y)} \norm[y]^{1/2} d^{\times}y \ll_{\epsilon} \lambda_W^{1/4+\epsilon} \Norm[W]_2, $$
	where the last inequality follows from Lemma \ref{SVKirL1Bd} \& \ref{SVKirL1ShBd} and \cite[\S 2.7.1]{Wu14}. Hence the desired bound is valid if $\norm[\tau] \geq T$ since $\lambda_W \geq 1+\norm[\tau]^2$. Let $h$ be a positive smooth function on $\ag{R}$ such that $0 \leq h \leq 1$, $h(y)=1$ for $\norm[y] \leq 1/2$ and $h(y)=0$ for $\norm[y] \geq 1$. By integration by parts, we get
\begin{align*}
	\extnorm{ \int_{\ag{R}^{\times}} (1-h)(\norm[y])W(y) \norm[y]^{1/2+i\tau} e^{-2\pi i Ty} d^{\times}y } &\ll \frac{1+\norm[\tau]}{\norm[T]} \int_{\ag{R}^{\times}} \extnorm{W(y)} + \extnorm{y \frac{d}{dy} W(y)} d^{\times}y \\
	&\ll_{\epsilon} \frac{1+\norm[\tau]}{\norm[T]} \lambda_W^{1/2+\epsilon} \Norm[W]_2,
\end{align*}
	where the last inequality follows again from Lemma \ref{SVKirL1Bd} \& \ref{SVKirL1ShBd} and \cite[\S 2.7.1]{Wu14}. Applying Lemma \ref{KIKirBd}, we get by integration by parts
\begin{align*}
	\extnorm{ \int_{\ag{R}^{\times}} h(\norm[y]) \widetilde{W}(y) \norm[y]^{1/2+i\tau} e^{-2\pi i Ty} d^{\times}y } &\ll \frac{1}{\norm[T]} \int_{\norm[y] \leq 1} \extnorm{ \frac{d}{dy}\left( h(\norm[y]) \widetilde{W}(y) \norm[y]^{-1/2+i\tau} \right) } dy \\
	&\ll_{\epsilon} \frac{1+\norm[\tau]}{\norm[T]} \lambda_W^{5/4+\epsilon} \Norm[W]_2.
\end{align*}
	Lemma \ref{FourErd} implies if $\norm[T] > \norm[\tau]/2$
	$$ \extnorm{ \int_{\ag{R}^{\times}} h(\norm[y]) \norm[y]^{1+2i\tau} e^{-2\pi i Ty} d^{\times}y } \ll \norm[\tau]^{1/2} \norm[T]^{-1} + (\norm[T] - \norm[\tau]/2)^{-1}. $$
	We obviously have for any $N \in \ag{N}$
	$$ \extnorm{ \int_{\ag{R}^{\times}} h(\norm[y]) \norm[y] e^{-2\pi i Ty} d^{\times}y } \ll_N \norm[T]^{-N}. $$
	Taking into account the bounds for $a_{\pm}(W)$ in Lemma \ref{KIKirBd}, we get the desired bound for $\norm[\tau] \leq \norm[T]$.
\end{proof}
		
		\subsubsection{Upper Bounds for Truncation}

\begin{lemma}
	Suppose $\F = \ag{R}$. Conditions are as in Lemma \ref{LocLBdReal}. Let $s \in \ag{C}$ with $\Re s = \sigma > -1/2$ varying in a compact interval included in the real line.
\begin{itemize}
	\item[(1)] For the option (A), we have
	$$ \extnorm{ \int_{\ag{R}^{\times}} \phi(y) e^{-2\pi i T_vy} \chi(y) \norm[y]^s d^{\times} y } \ll (1+\norm[s]) \norm[\mu]^{-1/2}. $$
	\item[(2)] Suppose $\pi = \pi(\norm^{i\tau/2}, \norm^{-i\tau/2})$ or $\pi(\norm^{i\tau/2} \sgn, \norm^{-i\tau/2})$ for some $0 \neq \tau \in \ag{R}$ upon twisting. For the option (B), we have
	$$ \extnorm{ \int_{\ag{R}^{\times}} W_0(y) e^{-2\pi i T_vy} \chi(y) \norm[y]^s d^{\times}y } \ll_{\epsilon} (1+\norm[\tau]+\norm[s]) \norm[\mu]^{-1/2}, \quad \text{if } \sigma = -1/2+\epsilon; $$
	$$ \extnorm{ \int_{\ag{R}^{\times}} W_0(y) e^{-2\pi i T_vy} \chi(y) \norm[y]^s d^{\times}y } \ll_{\epsilon} (1+\norm[\tau])^{1/2+\epsilon}(1+\norm[\tau]+\norm[s]) \norm[\mu]^{-1/2}, \quad \text{if } \sigma = 1/2+\epsilon. $$
	\item[(3)] Suppose $\pi = \pi(\mu_1,\mu_2)$ with $\mu_1 \mu_2^{-1}(t) = t^p \sgn(t)$ for some integer $p > 0$. For the option (B), we have
	$$ \extnorm{ \int_{\ag{R}^{\times}} W_0(y) e^{-2\pi i T_vy} \chi(y) \norm[y]^s d^{\times}y } \ll (1+\norm[s]) p^{\sigma+3/4} \norm[\mu]^{-1/2}. $$
\end{itemize}
\label{LocUBdReal}
\end{lemma}
\begin{proof}
	(1) We would like to say it's ``precisely'' \cite[Corollary 4.3]{Wu14} as we did in \cite[\S 6.1]{Wu14}, but indeed \cite[Corollary 4.3]{Wu14} did not deal with complex exponent. Instead, we can apply Lemma \ref{Erd} to ``$\phi(x)=\phi(e^x) \norm[e^x]^s$'' and see that the LHS is bounded as $\ll \norm[\mu]^{-1/2} + \norm[s] \norm[\mu]^{-1/2}$.
	
\noindent (2) We apply Lemma \ref{Erd} as in the proof of Lemma \ref{LocLBdReal} (2), but to $W_0(y) \norm[y]^s$ instead of $W_0(y)$ and with $N=1$ instead of $N=2$. The relevant norms $\extNorm{A^n.(W_0(y) \norm[y]^s)}_1$ for $n=0,1$ are bounded using Lemma \ref{SVKirL1Bd}, \ref{SVKirL1ShBd} together with \cite[\S 2.7.1]{Wu14}.

\noindent (3) We argue as in (2). The bound follows from the explicit computation
	$$ \extNorm{A^n.(W_0(y) \norm[y]^s)}_1 \ll_{n,\sigma} (1+\norm[s]) p^{\sigma - 1/4 + n}. $$ 
\end{proof}
\begin{lemma}
	Suppose $\F=\ag{C}$, Conditions are as in Lemma \ref{LocLBdCp}. Let $s \in \ag{C}$ with $\Re s = \sigma > -1/2$ varying in a compact interval included in the real line.
\begin{itemize}
	\item[(1)] For the option (A), we have
	$$ \extnorm{ \int_{\ag{C}^{\times}} \phi(y) e^{-2\pi i (T_vy + \overline{T_vy})} \chi(y) \norm[y]_{\ag{C}}^s d^{\times} y } \ll_{\epsilon} \max(\norm[\mu], \norm[m])^{-1} +\norm[s]^{4+\epsilon} \max(\norm[\mu], \norm[m])^{-2}. $$
	\item[(2)] Suppose $\pi = \pi(\norm_{\ag{C}}^{i\tau/2}, \norm_{\ag{C}}^{-i\tau/2})$ for some $0 \neq \tau \in \ag{R}$ upon twisting. For the option (B), we have
\begin{align*}
	&\quad \extnorm{ \int_{\ag{C}^{\times}} W_0(y) e^{-2\pi i (T_vy + \overline{T_vy})} \chi(y) \norm[y]^s d^{\times}y } \\
	&\ll_{\epsilon} \left\{ \begin{matrix} \norm[\tau]^{-1/2+2\epsilon} \max(\norm[\mu], \norm[m])^{-1} + (1+\norm[\tau]+\norm[s])^{4+\epsilon} \max(\norm[\mu], \norm[m])^{-2} & \text{if } \sigma = -1/2+\epsilon; \\ \norm[\tau]^{5/3+\epsilon} \max(\norm[\mu], \norm[m])^{-1} + (1+\norm[\tau])^{1+\epsilon}(1+\norm[\tau]+\norm[s])^{4+\epsilon} \max(\norm[\mu], \norm[m])^{-2} & \text{if } \sigma = 1/2+\epsilon. \end{matrix} \right.
\end{align*}
	\item[(3)] Suppose $\pi = \pi(\alpha^N,\alpha^{-N})$ or $\pi = \pi(\alpha^{N+1},\alpha^{-N})$ for some integer $N > 0$. For the option (B), we have
\begin{align*}
	&\quad \extnorm{ \int_{\ag{C}^{\times}} W_0(y) e^{-2\pi i (T_vy + \overline{T_vy})} \chi(y) \norm[y]^s d^{\times}y } \\
	&\ll_{\epsilon} \left\{ \begin{matrix} N^{-1/4+2\epsilon} \max(\norm[\mu], \norm[m])^{-1} + (1+N+\norm[s])^{4+\epsilon} \max(\norm[\mu], \norm[m])^{-2} & \text{if } \sigma = -1/2+\epsilon; \\ N^{7/4+\epsilon} \max(\norm[\mu], \norm[m])^{-1} + N^{1+\epsilon}(1+N+\norm[s])^{4+\epsilon} \max(\norm[\mu], \norm[m])^{-2} & \text{if } \sigma = 1/2+\epsilon. \end{matrix} \right.
\end{align*}\end{itemize}
\label{LocUBdCp}
\end{lemma}
\begin{proof}
	(1) As in the real case, we apply Lemma \ref{BesselErd} to ``$\ag{R} \times \ag{R}/\ag{Z} \ni (x, \alpha) \mapsto \phi(e^x) e^{2sx}$'' and see that the LHS is bounded as $\ll_{\epsilon} \norm[\mu]^{-1} + \norm[s]^{4+\epsilon} \norm[\mu]^{-2}$.

\noindent (2) We apply Lemma \ref{BesselErd} as in the proof of Lemma \ref{LocLBdCp} (2), but to $W_0(y) \norm[y]_{\ag{C}}^s$ instead of $W_0(y)$. The relevant norms $\extNorm{A^n.(W_0(y) \norm[y]_{\ag{C}}^s)}_1, \extNorm{A^n.(W_0(y) \norm[y]_{\ag{C}}^s)}_2$ for $0 \leq n \leq 4$ are bounded using Lemma \ref{SVKirL1Bd}, \ref{SVKirL1ShBd} together with \cite[\S 2.7.2]{Wu14}.

\noindent (3) We argue exactly as in (2). Note that this bound should be weaker than the one obtained by exploiting the relevant Bessel-$K$ functions, as what we have done for the real case.
\end{proof}

\section{Global Estimations}

	From now on, we restrict to the option (A) given in \S \ref{ArchChoice}. It is easy to check that all the following arguments are valid for the option (B) under the \emph{Assumptions (A) \& (B)}.

	\subsection{Refinement for Truncation}
	
	Recall (\cite[\S 6.1]{Wu14}) $h_0 \in \Cont^{\infty}(\ag{R}_+)$ such that $0 < h_0 < 1, h_0 \mid_{(0,1]} = 1$ and for any $X > 0$ we denote $h_{0,X}(t) := h_0(t/X)$.
\begin{lemma}
	Let $h(t) := h_{0, B} - h_{0, A}$ with $A = \Cond(\chi)^{-\kappa-1}, B = \Cond(\chi)^{\kappa-1}$ where $0 < \kappa < 1$ is to be optimized later. Then we have for some constant $C > 0$
\begin{align*}
	&\quad \extnorm{ \zeta(1/2, \varphi, \chi) - \int_{\F^{\times} \backslash \ag{A}^{\times}} h(\norm[y]_{\ag{A}}) \varphi(a(y)) \chi(y) d^{\times}y } \\
	&\ll_{\F, \epsilon} (\Cond(\pi) \Cond(\chi))^{\epsilon} \Cond(\pi_{\infty})^C \Cond_{\fin}(\pi, \chi)^{\theta} \Cond(\pi)^{1/2} \Cond(\chi)^{-\kappa/2}.
\end{align*}
\label{TruncEst}
\end{lemma}
\begin{proof}
	Mellin inversion formula gives
	$$ \int_{\F^{\times} \backslash \ag{A}^{\times}} h_{0,A}(\norm[y]_{\ag{A}}) \varphi(a(y)) \chi(y) d^{\times}y = \int_{\Re s = 1/2-\epsilon} A^s \Mellin{h_0}(s) \zeta(1/2-s, \varphi, \chi) \frac{ds}{2\pi i}, $$
	$$ \int_{\F^{\times} \backslash \ag{A}^{\times}} (1-h_{0,B})(\norm[y]_{\ag{A}}) \varphi(a(y)) \chi(y) d^{\times}y = -\int_{\Re s = 1/2+\epsilon} B^{-s} \Mellin{h_0}(-s) \zeta(1/2+s, \varphi, \chi) \frac{ds}{2\pi i} $$
where $\Mellin{h_0}(s)$ is (the analytic continuation of) the Mellin transform of $h_0$. Recollecting Lemma \ref{LocUBdFinite}, \ref{LocUBdReal}, \ref{LocUBdCp} and applying the convex bound for $L(s,\pi \otimes \chi)$ we get on the respective vertical lines
\begin{align*}
	\extnorm{ \zeta(1/2-s, \varphi, \chi) } &= \extnorm{ L(1/2-s, \pi \otimes \chi) \sideset{}{_v} \prod \ell_v(1/2-s, W_{\varphi,v}, \chi_v) } \\
	&\ll_{\F, \epsilon} \Cond(\pi \otimes \chi)^{1/2-\epsilon} (1+\norm[s])^{2d_{\F}} \Cond(\pi_{\infty})^C \Cond_{\fin}(\pi, \chi)^{\theta} \Cond(\chi)^{-1/2} \norm[L(1, \pi, \mathrm{Ad})]^{-1},
\end{align*}
\begin{align*}
	\extnorm{ \zeta(1/2+s, \varphi, \chi) } &= \extnorm{ L(1/2+s, \pi \otimes \chi) \sideset{}{_v} \prod \ell_v(1/2+s, W_{\varphi,v}, \chi_v) } \\
	&\ll_{\F, \epsilon} (1+\norm[s])^{2d_{\F}} \Cond(\pi_{\infty})^C \Cond(\chi)^{-1/2} \norm[L(1, \pi, \mathrm{Ad})]^{-1}.
\end{align*}
	Inserting the estimations
	$$ \Cond(\pi \otimes \chi) \leq \Cond(\pi) \Cond(\chi)^2, \quad \extnorm{ \Mellin{h_0}(s) } = \extnorm{ \frac{\Mellin{h_0^{(n)}}(s)}{s(s+1) \cdots (s+n-1)} } \ll_n (1+\norm[s])^{-n} $$
	and $\norm[L(1, \pi, \mathrm{Ad})] \gg_{\epsilon} \Cond(\pi)^{-\epsilon}$ due to \cite{HP94} and \cite[Lemma 3]{BH10}, we conclude.
\end{proof}

	We recall the bounds \cite[(6.1) \& (6.2)]{Wu14} as
\begin{align}
	\extnorm{\Mellin{h}(s)} &\ll \frac{\log(\Cond(\chi)) \extNorm{h_0^{(n)}}_{\infty} \int_1^2 t^{\Re s +n} d^{\times}t }{\extnorm{ (s+1) \cdots (s+n-1) }} \nonumber \\
	& \cdot \max ( (\Cond(\chi)^{\kappa -1})^{\Re s}, (\Cond(\chi)^{-\kappa -1})^{\Re s} ).
\label{MellinhBd}
\end{align}

	\subsection{Refinement for Constant Contribution}
	
\begin{lemma}
	Recall the notations giving (\ref{CstTermDef}). For $\Re s = \epsilon$, we have the estimation
	$$ \extnorm{ \zeta\left(s+1/2, \left( a\left( \frac{\varpi_{\vp_1}}{\varpi_{\vp_2}} \right).\varphi_0 \overline{a\left( \frac{\varpi_{\vp_1'}}{\varpi_{\vp_2'}} \right).\varphi_0} \right)_{\gp{N}} \right) } \ll_{\F, \epsilon} E^{-2} \prod_{\vp \in \vec{\vp}} \norm[\lambda_{\pi}(\vp)], $$
	where $\lambda_{\pi}(\vp)$ is the coefficient at $\vp$ of $L(s,\pi) = \sideset{}{_{\mathfrak{N}}} \sum \lambda_{\pi}(\mathfrak{N}) \Nr(\mathfrak{N})^{-s}$.
\label{CstEstRef}
\end{lemma}
\begin{proof}
	Since $\varphi_{0,\vp}$ is a new vector, the equation below \cite[(6.3)]{Wu14} is in fact a finite product, i.e.,
\begin{align*}
	&\quad \zeta\left(s+1/2, \left( a\left( \frac{\varpi_{\vp_1}}{\varpi_{\vp_2}} \right).\varphi_0 \overline{a\left( \frac{\varpi_{\vp_1'}}{\varpi_{\vp_2'}} \right).\varphi_0} \right)_{\gp{N}} \right) \\
	&= \Dis(\F)^{-1/2} \frac{L(s+1, \pi \times \overline{\pi})}{\zeta_{\F}(2s+2)} \prod_{v \mid \infty} \int_{\F_v^{\times}} \extnorm{ W_{0,v}(a(y)) }^2 \norm[y]_v^s d^{\times}y \cdot \prod_{\vp \in \vec{\vp}} \Sigma_{\vp},
\end{align*}
	with $\Sigma_{\vp}$ defined thereafter. The bound of the ratio of $L$-functions giving there is in fact independent of $\pi$ but only dependent of $\theta$, since it comes from a comparison with (Riemann) zeta function. Hence
	$$ \extnorm{ \frac{L(s+1, \pi \times \overline{\pi})}{\zeta_{\F}(2s+2)} } \ll_{\F, \epsilon} 1, \quad \Re s = \epsilon. $$
	The product over $v \mid \infty$ is absolutely bounded by $1$ (for the option (A), or can be bounded as $\Cond(\pi_{\infty})^{\epsilon}$ using Sobolev interpolation as in the proof of Lemma \ref{SVKirL1Bd} for the option (B)). The remaining part being bounded using \cite[(6.4) \& (6.5)]{Wu14} with ``$d_v = 0$'' by our definition of amplification measure (\ref{RegMeas}), we conclude.
\end{proof}
\begin{lemma}
	For any $\epsilon > 0$, we have
	$$ \sideset{}{_{\vp \in S(E)}} \sum \norm[\lambda_{\pi}(\vp)]^2 \ll_{\F, \epsilon} E(E \Cond(\pi))^{\epsilon}, \quad \sideset{}{_{\vp \in S(E)}} \sum \norm[\lambda_{\pi}(\vp)] \ll_{\F, \epsilon} E(E \Cond(\pi))^{\epsilon}. $$
\label{PNTRSRef}
\end{lemma}
\begin{proof}
	This is a refined version of \cite[Lemma 6.1]{Wu14}. We first use standard analytic argument (\cite[Remark 5.22]{IK04} for example) to establish
	$$ \sideset{}{_{E \leq \Nr(\idlN) \leq 2E}} \sum \norm[\lambda_{\pi}(\idlN)]^2 \ll_{\F, \epsilon} E + E^{1/2+\epsilon} \Cond(\pi)^{1+\epsilon}. $$
	The passage from the above bound to the first desired bound, well-known to experts as ``Iwaniec's trick'' (proof of \cite[Lemma 8]{BHMM16}), follows line by line the argument giving \cite[(19.16)]{DFI02}, replacing the divisor function $\tau_r$ with its counterpart for the number filed $\F$.
\end{proof}

	\subsection{Refinement for Cuspidal Contribution}
	
\begin{lemma}
	Recall the notations giving (\ref{CondPi'}), (\ref{CuspTermDef}). For $s \in i\ag{R}$, we have the estimation
	$$ \extnorm{ \zeta(s+1/2, n(T)e) } \ll_{\F,\epsilon} \Norm[T]^{-(1/2-\theta)+\epsilon} \extnorm{\frac{L(s+1/2,\pi')}{\sqrt{ L(1, \pi', \mathrm{Ad}) }}} (\dim \gp{K}_{\infty}.e_{\infty})^{1/2} \Cond_{\fin}[\pi,\chi]^{1/2}. $$
\label{CuspEstRef}
\end{lemma}
\begin{proof}
	This is a refinement of \cite[Lemma 6.5]{Wu14}. In fact, \cite[Lemma 6.8 \& 6.11 \& 6.12 \& 6.13]{Wu14} together with our normalization of local norms imply
\begin{align*}
	\extnorm{ \zeta(s+1/2, n(T)e) } &\ll_{\F,\epsilon} \extnorm{L(s+1/2,\pi')} \Norm[T]^{-(1/2-\theta)+\epsilon} (\dim \gp{K}_{\infty}.e_{\infty})^{1/2} \sideset{}{_{\vp: T_{\vp} \neq 0}} \prod (\dim \gp{K}_{\vp}.e_{\vp})^{1/2} \\
	&\quad \cdot \sqrt{ (2 L(1, \pi', \mathrm{Ad}))^{-1} \sideset{}{_{v \mid \infty}} \prod \zeta_v(1) \zeta_v(2)^{-1} }.
\end{align*}
	Note that $T_{\vp} \neq 0 \Leftrightarrow \cond(\chi_{\vp}) \neq 0$, at which $\dim \gp{K}_{\vp}.e_{\vp} \leq \Cond(\pi_{\vp})$.
\end{proof}
\begin{lemma}
	Let notations be as (\ref{CondPi'}), (\ref{CuspTermDef}) and $s \in i \ag{R}$. Denote by $\lambda_{e,\infty}$ the eigenvalue of
	$$ \Delta_{\infty} := \sideset{}{_{v \mid \infty}} \prod \left( - \Casimir_{\SL_2(\F_v)} - 2 \Casimir_{\gp{K}_v} \right) $$
	on the vector $e$ in $\pi'$. There are absolute constants $A, B > 0$ such that
	$$ \sideset{}{_{\pi'}} \sum \sideset{}{_{e \in \Bas(\pi')}} \sum \extnorm{ \frac{L(s+1/2, \pi')}{\sqrt{ L(1,\pi', \mathrm{Ad}) }}}^2 \lambda_{e,\infty}^{-A} \ll_{\epsilon} (1+\norm[s])^B (\Cond(\pi_{\fin}) E^4)^{1+\epsilon}. $$
\label{LindelAvg}
\end{lemma}
\begin{proof}
	This is simply \cite[Corollary 6.7]{Wu14}, rephrased by adding the harmonic weights. The detail of the proofs of \cite[Theorem 6.6 \& Corollary 6.7]{Wu14} will be worked out in a future paper, together with explicit constants $A,B$.
\end{proof}

	\subsection{Refinement for Eisenstein Contribution}
	
\begin{lemma}
	Recall the notations giving (\ref{CondXi}), (\ref{EisTermUDef}). For $s \in i\ag{R}$, we have the estimation
\begin{align*}
	&\quad \extnorm{ \zeta(s+1/2, n(T) \eis(i\tau, \Phi)) } \\
	&\ll_{\F,\epsilon} \Norm[T]^{-1/2+\epsilon} \extnorm{\frac{L(1/2+s+i\tau,\xi) L(1/2+s-i\tau,\xi^{-1})}{L(1+2i\tau,\xi^2)}} (\dim \gp{K}_{\infty}.\Phi_{\infty})^{1/2} \Cond_{\fin}[\pi,\chi]^{1/2}.
\end{align*}
\label{EisEstURef}
\end{lemma}
\begin{proof}
	This is a refinement of \cite[Lemma 6.14]{Wu14}. In fact, the principal series version of \cite[Lemma 6.8 \& 6.11 \& 6.12 \& 6.13]{Wu14} together with our normalization of local norms imply
\begin{align*}
	\extnorm{ \zeta(s+1/2, n(T) \eis(i\tau, \Phi)) } &\ll_{\F,\epsilon} \extnorm{\frac{L(1/2+s+i\tau,\xi) L(1/2+s-i\tau,\xi^{-1})}{L(1+2i\tau,\xi^2)}}\Norm[T]^{-1/2+\epsilon} \\
	&\quad \cdot (\dim \gp{K}_{\infty}.\Phi_{\infty})^{1/2} \sideset{}{_{\vp: T_{\vp} \neq 0}} \prod (\dim \gp{K}_{\vp}.\Phi_{\vp})^{1/2} \\
	&\quad \cdot \sideset{}{_{v \mid \infty}} \prod \zeta_v(1) \zeta_v(2)^{-1/2} \sideset{}{_{\vp: T_{\vp} \neq 0}} \prod \zeta_{\vp}(1) \zeta_{\vp}(2)^{-1/2}.
\end{align*}
	Note that $T_{\vp} \neq 0 \Leftrightarrow \cond(\chi_{\vp}) \neq 0$, at which $\dim \gp{K}_{\vp}.\Phi_{\vp} \leq \Cond(\pi_{\vp})$.
\end{proof}
\begin{lemma}
	Let notations be as (\ref{CondXi}), (\ref{EisTermUDef}) and $s \in i \ag{R}$. Denote by $\lambda_{\Phi,i\tau,\infty}$ the eigenvalue of $\Delta_{\infty}$ defined in Lemma \ref{LindelAvg} on the vector $\Phi_{i\tau}$ in $\pi(\xi \norm_{\ag{A}}^{i\tau}, \xi^{-1} \norm_{\ag{A}}^{-i\tau})$. There are absolute constants $A, B > 0$ such that
	$$ \sideset{}{_{\xi}} \sum \sideset{}{_{\Phi \in \Bas(\xi)}} \sum \int_{\ag{R}} \extnorm{\frac{L(1/2+s+i\tau,\xi) L(1/2+s-i\tau,\xi^{-1})}{L(1+2i\tau,\xi^2)}}^2 \lambda_{\Phi, i\tau,\infty}^{-A} \frac{d\tau}{4\pi} \ll_{\epsilon} (1+\norm[s])^B \Cond(\pi_{\fin})^{1+\epsilon}. $$
\label{LindelAvgEis}
\end{lemma}
\begin{proof}
	This is the counterpart of \cite[Corollary 6.7]{Wu14} for the continuous spectrum. But its proof is much simpler: it is a simple consequence of the convex bound.
\end{proof}
\begin{remark}
	In fact, the true size (true Lindel\"of on average) should be $\Cond(\pi_{\fin})^{1/2+\epsilon}$ on the RHS.
\end{remark}
\begin{lemma}
	Recall the notations giving (\ref{EisTermRDef}). There is an absolute constant $C>0$ such that
	$$ \extnorm{ \zeta^*(1 \pm i\tau, n(T).\eis(i\tau, \Phi)) } \ll_{\F, \epsilon} \Norm[T]^{-1+\epsilon} \lambda_{\Phi,i\tau,\infty}^C \Cond(\pi_{\fin})^{1/2}. $$
\label{EisEstRRef}
\end{lemma}
\begin{proof}
	Taking the decomposition (\ref{IntRepsLEisRes1}) \& (\ref{IntRepsLEisRes2}) into account, this is simply the global consequence of Lemma \ref{ExUpBdNArch} \& \ref{ExUpBdArch}.
\end{proof}

\section{Crude Bound of $\intL^4$-Norm}

	We would like to estimate the $\intL^4$-norm of a smooth unitary $\varphi \in \pi$, where $\pi$ is a cuspidal representation of $\GL_2$ over a number field $\F$ in terms of $\Cond(\pi_{\fin})$ and some polynomial dependence on $\Cond(\pi_{\infty})$. More generally, we shall estimate the $\intL^2$-norm of $\varphi_1 \overline{\varphi_2}$ for two smooth unitary $\varphi_1, \varphi_2 \in \pi$. To this end, we shall apply the Plancherel formula and need to estimate for each $e$ in an orthonormal basis of $\tau$, a cuspidal representation of $\PGL_2$ the inner product
	$$ \Pairing{\varphi_1 \overline{\varphi_2}}{e} = \int_{[\PGL_2]} \varphi_1(g) \overline{\varphi_2(g)} \overline{e(g)} dg. $$
	Ichino's triple product formula naturally applies. We need thus
\begin{itemize}
	\item control the $L$-factors, say by the convex bound;
	\item sum over $e$.
\end{itemize}
	For the first purpose, we need to control the size of the conductor.

\begin{lemma}
	Let $\pi$ and $\tau$ be cuspidal representations of $\GL_2$ over a number field $\F$. Assume that $\cond(\tau_{\vp}) \leq \cond(\pi_{\vp})$ at any $\vp < \infty$. The analytic conductor of $L(s, \mathrm{Ad} (\pi) \times \tau)$ is bounded as
	$$ \ll \Cond(\pi_{\infty})^2 \Cond(\tau_{\infty})^3 \Cond(\pi_{\fin})^2 \Cond(\tau_{\fin})^2 (\Cond(\pi_{\fin})^{\flat})^2, $$
	where we recall that $\Cond_{\fin}(\pi)^{\flat}$ is the product of $\Nr(\vp)$ for $\vp$ running over primes such that $\cond(\pi_{\vp}) > 0$.
\label{TripCond}
\end{lemma}
\begin{proof}
	We first recall the existing bounds of $\Cond(\mathrm{Ad}(\pi))$ in terms of $\Cond(\pi)$ in the literature. At an infinite place $v \mid \infty$ the local Langlands correspondence is available. We read from \cite[\S 4.1.1]{Wat02} that $\Cond(\mathrm{Ad} (\pi_v)) \ll \Cond(\pi_v)$. At a finite place $\vp$ such that $\cond(\pi_{\vp}) > 0$, a sharp bound is given by \cite[Proposition 2.5]{NPS13}, namely $\Cond(\mathrm{Ad}(\pi_{\vp})) \leq \Nr(\vp) \Cond(\pi_{\vp})$. Then we remark that at the infinite places, Rankin-Selberg $L$-functions is compatible with local Langlands correspondence by \cite[Theorem 2.1]{J09}, while at finite places we have the upper bound of the conductor of Rankin-Selberg $L$-functions by \cite{BuH97}. Together they yield
	$$ \Cond(\mathrm{Ad}(\pi) \times \tau) \ll \Cond(\pi_{\infty})^2 \Cond(\tau_{\infty})^3 \Cond(\pi_{\fin})^2 \Cond(\tau_{\fin})^2 (\Cond(\pi_{\fin})^{\flat})^2, $$
	which concludes the proof.
\end{proof}

\noindent For the second purpose, we need to make the estimation depend on some quantity $a(e)$ associated with $e$, whose sum is convergent. A natural candidate is $a(e) = \lambda_{e,\infty}^{-N}$, where $\lambda_{e,\infty}^{-N}$ is the eigenvalue of $e$ for $\Delta_{\infty}$ (defined in Lemma \ref{LindelAvg}), since we have Weyl's law. Precisely, we shall use the self-adjointness of $\Delta_{\infty}$ and write
	$$ \Pairing{\varphi_1 \overline{\varphi_2}}{e} = \lambda_{e,\infty}^{-N} \Pairing{\Delta_{\infty}^N(\varphi_1 \overline{\varphi_2})}{e}, $$
	then reduce the problem to bounding
	$$ \Pairing{(X.\varphi_1) \cdot \overline{ Y.\varphi_2 }}{e} $$
	for monomials $X,Y$ in the universal enveloping algebra of the Lie algebra of $\GL_2(\ag{A}_{\infty})$ of length depending linearly on $N$. To this end, the following extension of the decay of matrix coefficients to smooth vectors is convenient.
	
\begin{lemma}
	Let $\varphi_1, \varphi_2$ be two smooth vectors in a unitary irreducible representation $\pi$ of $\GL_2(\ag{R})$ or $\GL_2(\ag{C})$ with spectral parameter $\leq \theta$, where $\theta$ is any constant towards the Selberg conjecture. Then for some Sobolev norm $\Sob(\cdot)$ of degree bounded by some absolute constant and any $\epsilon > 0$, we have
	$$ \extnorm{ \Pairing{\pi(g).\varphi_1}{\varphi_2} } \ll_{\epsilon} \Sob(\varphi_1) \Sob(\varphi_2) \Xi(g)^{1-2\theta-\epsilon}, \quad \forall g \in \GL_2(\ag{R}) \text{ or } \GL_2(\ag{C}). $$
	We recall that $\Xi(g)$ is the corresponding Harisch-Chandra's function \cite[\S 5.2.1 \& 5.2.2]{CU05}.
\label{ExtDecayM}
\end{lemma}
\begin{proof}
	We decompose $\varphi_1$ resp. $\varphi_2$ into a weighted sum of unitary $\gp{K}$-isotypic vectors as
	$$ \varphi_1 = \sideset{}{_j} \sum a_j^{(1)} e_j, \quad \varphi_2 = \sideset{}{_j} \sum a_j^{(2)} e_j, $$
	where $e_j$ spans a $\gp{K}$-irreducible representation of dimension $d_j$. Then we have
	$$ \Pairing{\pi(g).\varphi_1}{\varphi_2} = \sideset{}{_{j_1,j_2}} \sum a_{j_1}^{(1)} \overline{ a_{j_2}^{(2)} } \Pairing{\pi(g).e_{j_1}}{e_{j_2}}. $$
	The decay of matrix coefficients \cite{CHH88} implies
	$$ \extnorm{ \Pairing{\pi(g).\varphi_1}{\varphi_2} } \ll_{\epsilon} \sideset{}{_{j_1,j_2}} \sum \extnorm{ d_{j_1}^{1/2} a_{j_1}^{(1)} } \extnorm{ d_{j_2}^{1/2} a_{j_2}^{(2)} } \Xi(g)^{1-2\theta-\epsilon}. $$
	By Cauchy-Schwartz inequality and Weyl's law we have
	$$ \sideset{}{_j} \sum \extnorm{ d_j^{1/2} a_j^{(i)} } \leq \left( \sideset{}{_j} \sum d_j^{-2} \right)^{1/2} \left( \sideset{}{_j} \sum d_j^3 \norm[a_j^{(i)}]^2 \right)^{1/2} \ll \Sob(\varphi_i), \quad i = 1,2. $$
\end{proof}

\begin{lemma}
	If $\varphi$ is a smooth vector in $\pi$ as above, whose Kirillov function is a fixed function in $\Cont_c^{\infty}(\ag{R}^{\times})$ or $\Cont_c^{\infty}(\ag{C}^{\times})$. Then as $d$ and $\pi$ vary, the Sobolev norm
	$$ \Sob_d(\varphi) \ll_d \Cond(\pi)^{O(d)}. $$
\label{SobBd}
\end{lemma}
\begin{proof}
	This is a direct consequence of the explicit description of the differential operators corresponding to elements of the Lie algebra of $\GL_2(\ag{R})$ or $\GL_2(\ag{C})$ given in \cite[Lemma 8.4]{Ve10}.
\end{proof}

\begin{proposition}
	Let $\varphi_1, \varphi_2$ be two smooth unitary vectors in a cuspidal representation $\pi$ of $\GL_2$ over a number field $\F$, new at every finite place. Assume that at any $v \mid \infty$, the Kirillov function of $\varphi_{1,v}$ resp. $\varphi_{2,v}$ is a fixed function in $\Cont_c^{\infty}(\ag{R}^{\times})$ or $\Cont_c^{\infty}(\ag{C}^{\times})$. Then there is an absolute constant $N>0$ such that
	$$ \Norm[ \varphi_1 \overline{\varphi_2} ]_2 \ll_{\F, \epsilon} \Cond(\pi_{\infty})^N \Cond(\pi_{\fin})^{5/2+\epsilon} (\Cond(\pi_{\fin})^{\flat})^{1/4}. $$
\label{L4CrudeBd}
\end{proposition}
\begin{proof}
	Take $\tau$ any cuspidal representation of $\PGL_2$ and $e \in \tau$ in an orthonormal basis. Let $S=S(\pi)$ be the union of the infinite places and the finite places at which $\pi$ is ramified. Ichino's formula \cite[Theorem 1.1]{Ic08} implies
\begin{align*}
	\extnorm{ \int_{[\PGL_2]} \varphi_1(g) \overline{\varphi_2(g)} \overline{e(g)} dg }^2 &= \frac{\zeta_{\F}^{S}(2)^2}{8} \cdot \frac{L^S(1/2, \pi \times \bar{\pi} \times \tau)}{L^S(1,\pi,\mathrm{Ad}) L^S(1,\bar{\pi},\mathrm{Ad}) L^S(1,\bar{\tau},\mathrm{Ad})} \cdot \sideset{}{_{v \in S}} \prod \\
	&\quad \int_{\PGL_2(\F_v)} \frac{\Pairing{\pi_v(g_v).\varphi_{1,v}}{\varphi_{1,v}}}{\Pairing{\varphi_{1,v}}{\varphi_{1,v}}} \cdot \overline{\frac{\Pairing{\pi_v(g_v).\varphi_{2,v}}{\varphi_{2,v}}}{\Pairing{\varphi_{2,v}}{\varphi_{2,v}}}} \cdot \overline{\frac{\Pairing{\pi_v(g_v).e_{v}}{e_{v}}}{\Pairing{e_{v}}{e_{v}}}} dg_v.
\end{align*}
	It is non-vanishing only if $\cond(\tau_{\vp}) \leq \cond(\pi_{\vp})$ and if $e_{\vp}$ is invariant by $\gp{K}_0[\vp^{\cond(\pi_{\vp})}]$ at every finite place $\vp$. At an infinite place $v \mid \infty$, let $\lambda_{e,v}$ be the eigenvalue of $e_v$ for the Laplacian operator $\Delta_v$ (defined in Lemma \ref{LindelAvg}). Then we have
\begin{align*}
	&\quad \int_{\PGL_2(\F_v)} \frac{\Pairing{\pi_v(g_v).\varphi_{1,v}}{\varphi_{1,v}}}{\Pairing{\varphi_{1,v}}{\varphi_{1,v}}} \cdot \overline{\frac{\Pairing{\pi_v(g_v).\varphi_{2,v}}{\varphi_{2,v}}}{\Pairing{\varphi_{2,v}}{\varphi_{2,v}}}} \cdot \overline{\frac{\Pairing{\pi_v(g_v).e_{v}}{e_{v}}}{\Pairing{e_{v}}{e_{v}}}} dg_v \\
	&= \lambda_{e,v}^{-N} \cdot \int_{\PGL_2(\F_v)} \frac{\Pairing{\pi_v(g_v).\varphi_{1,v}}{\varphi_{1,v}}}{\Pairing{\varphi_{1,v}}{\varphi_{1,v}}} \cdot \overline{\frac{\Pairing{\pi_v(g_v).\varphi_{2,v}}{\varphi_{2,v}}}{\Pairing{\varphi_{2,v}}{\varphi_{2,v}}}} \cdot \overline{\frac{\Pairing{\pi_v(g_v).\Delta_v^N e_{v}}{e_{v}}}{\Pairing{e_{v}}{e_{v}}}} dg_v \\
	&= \lambda_{e,v}^{-N} \cdot \sum_{X,Y}  \int_{\PGL_2(\F_v)} \frac{\Pairing{\pi_v(g_v).X.\varphi_{1,v}}{\varphi_{1,v}}}{\Pairing{\varphi_{1,v}}{\varphi_{1,v}}} \cdot \overline{\frac{\Pairing{\pi_v(g_v).Y.\varphi_{2,v}}{\varphi_{2,v}}}{\Pairing{\varphi_{2,v}}{\varphi_{2,v}}}} \cdot \overline{\frac{\Pairing{\pi_v(g_v).e_{v}}{e_{v}}}{\Pairing{e_{v}}{e_{v}}}} dg_v,
\end{align*}
	where $X,Y$ runs over a finite set of monomials in the universal enveloping algebra of the Lie algebra of $\GL_2(\ag{R})$ or $\GL_2(\ag{C})$ of degree $\leq 2N$ such that
	$$ \Delta_v^N (\phi_1 \phi_2) = \sideset{}{_{X,Y}} \sum X.\phi_1 \cdot Y.\phi_2. $$
	The extended decay of matrix coefficients Lemma \ref{ExtDecayM}, the Sobolev bound Lemma \ref{SobBd}, together with the explicit computation/estimation of the Harish-Chandra's $\Xi$-functions \cite[\S 5.2.1 \& 5.2.2]{CU05} \& \cite[Theorem 4.6.6]{Bu98} or \cite[\S 6.3.1]{Wu14} yield
	$$ \sideset{}{_{v \mid \infty}} \prod \int_{\PGL_2(\F_v)} \frac{\Pairing{\pi_v(g_v).\varphi_{1,v}}{\varphi_{1,v}}}{\Pairing{\varphi_{1,v}}{\varphi_{1,v}}} \cdot \overline{\frac{\Pairing{\pi_v(g_v).\varphi_{2,v}}{\varphi_{2,v}}}{\Pairing{\varphi_{2,v}}{\varphi_{2,v}}}} \cdot \overline{\frac{\Pairing{\pi_v(g_v).e_{v}}{e_{v}}}{\Pairing{e_{v}}{e_{v}}}} dg_v \ll_{\theta, \epsilon} \lambda_{e,\infty}^{-N} \Cond(\pi_{\infty})^{O(N+1)}. $$
	Similarly but more simply, we have
\begin{align}
	&\quad \sideset{}{_{\vp < \infty, \vp \in S}} \prod \int_{\PGL_2(\F_{\vp})} \frac{\Pairing{\pi_{\vp}(g_{\vp}).\varphi_{1,\vp}}{\varphi_{1,\vp}}}{\Pairing{\varphi_{1,\vp}}{\varphi_{1,\vp}}} \cdot \overline{\frac{\Pairing{\pi_{\vp}(g_{\vp}).\varphi_{2,\vp}}{\varphi_{2,\vp}}}{\Pairing{\varphi_{2,\vp}}{\varphi_{2,\vp}}}} \cdot \overline{\frac{\Pairing{\pi_{\vp}(g_{\vp}).e_{\vp}}{e_{\vp}}}{\Pairing{e_{\vp}}{e_{\vp}}}} dg_{\vp} \nonumber \\
	&\ll_{\theta} \sideset{}{_{\vp < \infty, \vp \in S}} \prod \Cond(\pi_{\vp})^2 d_{k,\vp} \leq \Cond(\pi_{\fin})^3, \label{CrudeFinEst}
\end{align}
	where $d_{k,\vp} = \dim \gp{K}_{\vp}.e_{\vp}$ and in the worst case runs over integers
	$$ d_{0,\vp} = 1, \quad d_{1,\vp} = \Nr(\vp), \quad d_{k,\vp} = \Nr(\vp)^k - \Nr(\vp)^{k-2}, 2 \leq k \leq \cond(\pi_{\vp}). $$
	Note that
	$$ L^S(s, \pi \times \bar{\pi} \times \tau) = L^S(s,\tau) L^S(s, \mathrm{Ad}(\pi) \times \tau). $$
	We shall apply the convex bound for $L^S(s, \mathrm{Ad}(\pi) \times \tau)$ together with the bound of the conductor Lemma \ref{TripCond}, and the lower bound of the adjoint $L$-functions at $1$ obtained in \cite{HP94}, generalized to the number field case in \cite[Lemma 3]{BH10}. It follows that
\begin{align*}
	&\quad \frac{\zeta_{\F}^{S}(2)^2}{8} \cdot \frac{L^S(1/2, \pi \times \bar{\pi} \times \tau)}{L^S(1,\pi,\mathrm{Ad}) L^S(1,\bar{\pi},\mathrm{Ad}) L^S(1,\bar{\tau},\mathrm{Ad})} \\
	&\ll_{\epsilon} \frac{L^S(1/2,\tau)}{L^S(1,\bar{\tau},\mathrm{Ad})} \Cond(\pi_{\infty})^{1/2} \Cond(\tau_{\infty})^{3/4} \Cond(\pi_{\fin})^{1/2} \Cond(\tau_{\fin})^{1/2} (\Cond(\pi_{\fin})^{\flat})^{1/2} (\Cond(\pi)\Cond(\tau))^{\epsilon}.
\end{align*}
	
\noindent Summing over $e, \tau$ and using the average Lindel\"of bound \cite[Theorem 6.6]{Wu14}, we get
	$$ \extNorm{ \Proj_{\mathrm{cusp}}(\varphi_1 \overline{\varphi_2}) }^2 \ll_{\theta, \epsilon, N} \Cond(\pi_{\infty})^{O(N+3)} \Cond(\pi_{\fin})^{5+\epsilon} (\Cond(\pi_{\fin})^{\flat})^{1/2}, $$
where $\Proj_{\mathrm{cusp}}$ is the orthogonal projection onto the cuspidal spectra. Similar argument works the same (and simpler) for the continuous spectra and the one dimensional spectra. We thus conclude the proof.
\end{proof}

\begin{remark}
	Implicitly in the above proof, we have used the explicit local decomposition of the Haar measure on $\PGL_2(\F_v) = \gp{K}_v \gp{A}_v^+ \gp{K}_v, dg_v = \delta_v(t) d \kappa_1 dt d \kappa_2$, where
	$$ \gp{A}_v^+ = \left\{ \begin{matrix} \{ a(t): t \geq 1 \} & \text{if } \F_v = \ag{R} \text{ or } \ag{C}, \\ \{ a(\varpi_{\vp}^n): n \in \ag{N} \} & \text{if } v=\vp < \infty, \end{matrix} \right. $$
	$$ \text{and} \quad \delta_v(t) = \left\{ \begin{matrix} 1-t^{-2} & \text{if } \F_v = \ag{R}, \ag{C}, \\ 1_{n=0} + (q^n+q^{n-1}) 1_{n \geq 1} & \text{if } v=\vp < \infty, t=\varpi_{\vp}^n, \end{matrix} \right. $$
	We also note that the computation of $\delta_v$ in the real case is given in \cite[(7.22)]{KL72}, from which the complex case follows since the restriction of the Haar measure on $\PGL_2(\ag{C})$ onto $\PGL_2(\ag{R})$ must coincide with the one of the later.
\end{remark}

\begin{remark}
	The above estimation is really crude in that the bounds for local factors at finite places (\ref{CrudeFinEst}) uses general decay of matrix coefficients. To be convinced that this is far from being its true size, one can specialize to the case $\tau$ is Eisenstein and compare (\ref{CrudeFinEst}) with \cite[Corollary 2.8]{NPS13}. In general, we expect the right hand side of (\ref{CrudeFinEst}) to be replaced by $\Cond(\pi_{\fin})^{\epsilon}$ for any small $\epsilon > 0$.
\label{ReasonCrude}
\end{remark}

\section{Appendix}

	\subsection{Some Asymptotic Analysis}
	
		\subsubsection{One Dimensional Case}
		
\begin{lemma}
	Let $S(x)$ be a smooth real valued function on $\ag{R}$, admitting a stationary point $x_0$ of order $m-1 \in \ag{N}$ (c.f. \cite[p.p. 52]{Er56}) in the interval $(a,b)$. For simplicity, we assume $x_0$ is the unique stationary point. Let $\phi(x)$ be a smooth function such that for any $n \in \ag{N}$
	$$ \lim_{x \to a, b} (L^n \phi)(x) = 0, \quad \text{where} \quad L:= \frac{d}{dx} \circ \frac{1}{S'(x)}. $$
	Then for $\mu \in \ag{R}$, as $\norm[\mu] \to \infty$, we have for any $N \in \ag{N}$
	$$ \extnorm{ \int_a^b \phi(x) e^{i\mu S(x)} dx - \frac{1}{m}\sum_{n=0}^{N-1} \frac{\Gamma((n+1)/m)}{n!} k^{(n)}(0) e^{\frac{\varepsilon i \pi (n+1)}{2m}} \frac{e^{i\mu S(x_0)}}{\norm[\mu]^{(n+1)/m}} } \ll \frac{\Gamma(N/m)}{(N-1)!} \frac{1}{\norm[\mu]^{N/m}} \sum_{n=0}^N \Norm[\phi^{(n)}]_1, $$
	where $\varepsilon = \sgn(\mu S^{(m+1)}(x_0))$. The implied constant depends only on the function $x \mapsto S(x+x_0)-S(x_0)$. The function $k(x)$ depends only on $x \mapsto S(x+x_0)-S(x_0)$ and $\phi$. In particular,
	$$ k(0) = \left( \frac{\norm[S^{(m)}(0)]}{m!} \right)^{-\frac{1}{m}} \phi(x_0). $$
\label{Erd}
\end{lemma}
\begin{proof}
	This is a special case of the discussion in \cite[\S 2.9]{Er56} for integral order stationary points. It follows in particular from \cite[\S 2.9 (10) \& (17) \& (20)]{Er56}.
\end{proof}
\begin{remark}
	The validity of the above lemma extends to \emph{tempered phase function} in the sense of Definition \ref{TemPhase} below if either $a$ or $b$ or both are infinite, with extra error bound of smaller order.
\end{remark}
\begin{lemma}
	If $\phi(t)$ is $N+1$ times continuously differentiable in a finite interval $[0,b]$ with $\phi^{(n)}(b)=0$ for $0 \leq n \leq N$ and $\lambda \in \ag{C}$ with $\Re \lambda \in (0,1]$, then as $x \to \infty$
	$$ \extnorm{ \int_0^b \phi(t) t^{\lambda-1} e^{ixt} dt - \sum_{n=0}^N \frac{\Gamma(n+\lambda)}{n!} e^{\sgn(x)i\pi(n+\lambda)/2} \phi^{(n)}(0) \norm[x]^{-(n+\lambda)} } \leq \frac{\Gamma(N+\Re \lambda)}{N! \cdot \norm[x]^{N+1}} e^{\frac{\pi}{2}\norm[\Im \lambda]} \int_0^b \extnorm{ \phi^{(N+1)}(t) } dt. $$
	If moreover, $\phi^{(N+1)}$ vanishes identically on $[0,\delta]$ for some $0<\delta \leq b$ and $\norm[x] \geq T_0:= \norm[\Im \lambda] / \delta$, then we can replace the right hand side by
	$$ \frac{\Gamma(N+\Re \lambda)}{N! \cdot (\norm[x]-T_0)^{N+1}} \int_0^{\beta} \extnorm{ \phi^{(N+1)}(t) } dt. $$
\label{FourErd}
\end{lemma}
\begin{proof}
	We may assume $x>0$. The case $\lambda \in \ag{R}$ is a special case of the discussion in \cite[\S 2.8, pp. 47-49]{Er56}. In our case, we need to modify the bound of $u^{\lambda-1}$ in
	$$ h_{-n-1}(t) = \frac{(-1)^{n+1}}{n!} \int_t^{i\infty} (u-t)^n u^{\lambda-1} e^{ixu} du, $$
where the path of integration is taken as the ray $u=t+i\tau, \tau \geq 0$. We have
	$$ \norm[u^{\lambda-1}] = e^{\Re(\lambda-1) \log (\sqrt{t^2+\tau^2}) - (\Im \lambda) \arctan (\tau/t)} \leq \tau^{\Re(\lambda-1)} e^{\frac{\pi}{2} \norm[\Im \lambda]} \Rightarrow \norm[h_{-n-1}(t)] \leq \frac{\Gamma(n+\Re \lambda)}{n! \cdot x^{n+1}} e^{\frac{\pi}{2}\norm[\Im \lambda]} $$
implying the first estimation. For the ``moreover'' part, we note that the function
	$$ S(\tau) := - (\Im \lambda) \arctan (\tau/t) - T_0 \tau $$
verifies $S(0)=0, S'(\tau) \leq \norm[\Im \lambda] / t - T_0 \leq 0$ if $t \geq \delta$. Hence $S(\tau) \leq 0$ and we have alternatively
	$$ \norm[u^{\lambda-1}] \leq \tau^{\Re(\lambda-1)} e^{T_0 \tau} \Rightarrow \norm[h_{-n-1}(t)] \leq \frac{\Gamma(n+\Re \lambda)}{n! \cdot (x-T_0)^{n+1}} $$
implying the second estimation.
\end{proof}

		\subsubsection{Higher Dimensional Case}
		
\begin{definition}
	Let $S \in \Cont^{\infty}(\ag{R}^n)$ be a smooth real valued function. Associated to it there are $n$ weight functions $\omega_i$ and $n$ differential operators $L_i^*$ defined by
	$$ \omega_i(\vec{x}) = \frac{\frac{\partial S}{\partial x_i}}{\lVert \nabla S \rVert^2}, L_i^* = \frac{\partial }{\partial x_i} \circ \omega_i, 1\leq i \leq n. $$
	If $\nabla S(\vec{x}) = \vec{0}$ has only finitely many solutions in $\ag{R}^n$, and if for any index $\vec{\alpha} \in \ag{N}^n$
	$$ \limsup_{\vec{x} \to \infty} \norm[ \omega_i^{(\vec{\alpha})}(\vec{x}) ] < \infty, $$
i.e., any partial derivative of $\omega_i$ is bounded away from the critical points of $S(\vec{x})$, we call $S(\vec{x})$ a \emph{tempered phase function}.
\label{TemPhase}
\end{definition}
\begin{remark}
	If $\phi \in \Cont_0^{\infty}(\ag{R}^n)$, i.e, $\lim_{\vec{x} \to \infty} \phi^{(\vec{\alpha})}(\vec{x}) = 0$ for any index $\vec{\alpha} \in \ag{N}^n$, then for any word in $n$ variables $P$ we have
	$$ \lim_{x_k \to \pm \infty} |\omega_i(x) P(L_1^*, \cdots, L_n^*) \phi(x)| = 0, 1 \leq i,k \leq n. $$
\end{remark}
\begin{lemma}
	Let $S \in \Cont^{\infty}(\ag{R}^n)$ be a tempered phase function and $\phi \in \Cont_0^{\infty}(\ag{R}^n) \cap W^{\infty,1}(\ag{R}^n) \cap W^{\infty,2}(\ag{R}^n)$, i.e., $\phi$ lie in the infinite order Sobolev space both for $\intL^1$ and $\intL^2$-norms. Consider the oscillatory integral for $\mu \in \ag{R}$
	$$ I(\mu, \phi, S) = \int_{\ag{R}^n} \phi(x)e^{i \mu S(x)} dx. $$
	Suppose $x_0 \in \ag{R}^n$ such that $\nabla S(x_0) = \vec{0}, \det \nabla^2 S(x_0) \neq 0$ and $\nabla S(x) \neq \vec{0}$ for any $x \neq x_0$ in the support of $\phi$, i.e., $x_0$ is the unique stationary point in the sense of \cite[\S 3.5]{EZ03}. Then there exist for $k \in \ag{N}$ differential operators $A_{2k}(x,D)$ of order less than or equal to $2k$, such that for any $N \in \ag{N}, \epsilon > 0$
\begin{align*}
	&\quad \extnorm{ I(\mu, \phi, S) - \left( \sum_{k=0}^{N-1} (A_{2k}(x,D) \phi)(x_0) \mu^{-(k+n/2)} \right) e^{i\mu S(x_0)} } \\
	&\ll_{N,\epsilon} \left\{ \sum_{\norm[\vec{\alpha}] \leq N+\lceil n/2 \rceil} \Norm[ \phi^{(\vec{\alpha})} ]_1 + \left( \sum_{\norm[\vec{\alpha}] \leq 2N+n} \Norm[ \phi^{(\vec{\alpha})} ]_2 \right)^{1-\epsilon} \left( \sum_{\norm[\vec{\alpha}] \leq 2N+n+1} \Norm[ \phi^{(\vec{\alpha})} ]_2 \right)^{\epsilon} \right\} \norm[\mu]^{-(N+n/2)},
\end{align*}
	where both $A_{2k}(x,D)$ and the implied constant in the last inequality depend only on the function $x \mapsto S(x+x_0)-S(x_0)$. In particular,
	$$ (A_0(x,D) \phi)(x_0) = \left( \frac{\pi}{2} \right)^{n/2} \extnorm{ \det \nabla^2 S (x_0) }^{-1/2} e^{i \frac{\pi}{4} \sgn \left( \mu \nabla^2 S(x_0) \right) } \phi(x_0). $$
\label{StatPhase}
\end{lemma}
\begin{proof}
	This is the $n$-dimensional version of Lemma \ref{Erd} with order $m=2$. It is also a variant of \cite[Theorem 3.14]{EZ03} with two differences:

\noindent (1) The class of $\phi$ is enlarged. One can easily check that the definition of \emph{tempered phase function} ensures the validity of every integration by parts in the proof of \cite[Lemma 3.12]{EZ03}, as well as the subsequent bounds of integral in terms of $\intL^1$-norms.

\noindent (2) The bound of the error term (in terms of $\intL^2$-norms instead of $\intL^{\infty}$-norms) is different. It is obtainable by replacing \cite[Lemma 3.5]{EZ03} with
\begin{align*}
	\int_{\ag{R}^n} \norm[\hat{u}(\vec{x})] d\vec{x} &\leq \left( \int_{\ag{R}^n} \norm[\hat{u}(\vec{x})]^2 (1+\Norm[\vec{x}]^2)^{n+\epsilon} d\vec{x} \right)^{1/2} \left( \int_{\ag{R}^n} (1+\Norm[\vec{x}]^2)^{-(n+\epsilon)} d\vec{x} \right)^{1/2} \\
	&\ll_{\epsilon} \left( \int_{\ag{R}^n} \norm[\hat{u}(\vec{x})]^2 (1+\Norm[\vec{x}]^2)^n d\vec{x} \right)^{(1-\epsilon)/2} \left( \int_{\ag{R}^n} \norm[\hat{u}(\vec{x})]^2 (1+\Norm[\vec{x}]^2)^{n+1} d\vec{x} \right)^{\epsilon/2}
\end{align*}
	and the isometry of Fourier transform in terms of $\intL^2$-norms.
\end{proof}
\begin{remark}
	Although we stated our result with $\ag{R}^n$, it is easy to verify its validity for $\ag{R}^n \times (\ag{R}/(2\pi \ag{Z}))^m$. In the later case, it suffices to modify the definition of temperedness as
	$$ \limsup_{\vec{x} \to \infty} \norm[ \omega_i^{(\vec{\alpha})}(\vec{x}, \vec{y}) ] < \infty, \quad \vec{x} \in \ag{R}^n, \vec{y} \in (\ag{R}/(2\pi \ag{Z}))^m. $$
	In fact, the localization argument around $x_0$ works the same way, and the rapid decay part with integration by parts works even simpler at the compact component.
\label{TemPhaseExt}
\end{remark}

		\subsubsection{Some Asymptotic Related to Bessel Functions}
	
	We denote by $J_m$ resp. $K_m$ the Bessel functions of the first kind resp. the modified Bessel functions of the second kind of order $m \in \ag{N}$.
\begin{lemma}
	Let $m \in \ag{N}, u \in [0,1], r \geq r_0 > 0$ and $x \gg m^2$, where $r_0$ is a constant, then we have
	$$ \extnorm{ K_m((u \pm ri)x) } \ll \sqrt{ \frac{\pi}{2} } (r_0^2+u^2)^{-1/4} x^{-1/2} e^{-ux}. $$
\label{BesselKLargeBd}
\end{lemma}
\begin{proof}
	Specializing the relation between the Bessel-$K$ functions and the Hankel functions \cite[(5.3) \& (5.4)]{Wo77} to our case, we get
	$$ H_m^{(1)}(i(u-ri)x) = \frac{2}{\pi i} e^{-\frac{im\pi}{2}} K_m((u-ri)x), \quad H_m^{(2)}(-i(u+ri)x) = -\frac{2}{\pi i} e^{\frac{im\pi}{2}} K_m((u+ri)x). $$
	The asymptotic expansions of Hankel functions are obtained in \cite[\S \Rmnum{7}.7.2]{Wat44} with error bounds. For example for $H_m^{(1)}$, we can take $\beta = 0, \delta = \pi/2$ and deduce $A_p = 1$ in the cited discussion, yielding the following bound
\begin{align*}
	& \extnorm{ K_m((u-ri)x) - \sqrt{ \frac{\pi}{2(u-ri)x} } e^{-(u-ri)x} \sum_{n=0}^{p-1} \frac{(1/2-m)_n (1/2+m)_n}{n! (2(u-ri)x)^n} } \\
	&\leq \sqrt{ \frac{\pi}{2\norm[(u-ri)x]} } e^{-ux} \extnorm{ \frac{(1/2-m)_p (-1/2-m)_p}{p! (2(u-ri)x)^p} },
\end{align*}
	where $p \geq m$ is any integer. Choosing $p=n$ and taking into account the bounds
	$$ \extnorm{ \frac{(1/2-m)_n (1/2+m)_n}{n! (2(u-ri)x)^n} } \leq \frac{C^m}{m!}, \quad \extnorm{ \frac{(1/2-m)_m (-1/2-m)_m}{m! (2(u-ri)x)^m} } \leq \frac{C^m}{m!} $$
	for some constant $C$ depending on $r_0$ and $x \gg m^2$, we conclude the proof with implied constant $e^C$.
\end{proof}
\begin{lemma}
	Suppose $\phi(r)$ is $N$ times continuously differentiable in $[0,1]$ with $\phi^{(n)}(1)=0$ for $0 \leq n \leq N-1$. Suppose also that $\phi^{(N)}(r)=0$ for $0 \leq r \leq r_0$ for some constant $r_0 \in (0,1]$. Let $x, \lambda \in \ag{R}$ such that $x \geq 1+\max(T_0, m^2)$ where $T_0 := \norm[\lambda]/r_0$. Writing as in \cite[(3.1)]{Wo77}
	$$ \Lambda_m(\alpha) := 2^{\alpha-1} \Gamma(\frac{\alpha+1+m}{2}) \Gamma(\frac{\alpha+1-m}{2}), $$
	we then have
\begin{align*}
	&\quad \extnorm{ \int_0^1 \phi(r) r^{i\lambda} J_m(rx) dr + \sum_{n=0}^{N-1} \phi^{(n)}(0) \frac{i^{n+m} e^{\frac{\pi}{2}\lambda} + i^{-(n+m)} e^{-\frac{\pi}{2}\lambda}}{\pi x^{1+n+i\lambda}} \frac{\Lambda_m(n+i\lambda)}{n!} } \\
	&\ll x^{-1/2}(x-T_0)^{-N} \int_0^1 \norm[ \phi^{(N)}(r) ] dr.
\end{align*}
\label{BesselErd}
\end{lemma}
\begin{proof}
	In view of the decomposition \cite[(6.15)]{Wo77}
	$$ J_m(x) = \frac{i^{-(m+1)}}{\pi} K_m(\frac{x}{i}) + \frac{i^{m+1}}{\pi} K_m(ix), $$
	we are reduced to estimating
	$$ \int_0^1 \phi(r) r^{i\lambda} K_m(-irx) dr, \quad \text{resp.} \quad \int_0^1 \phi(r) r^{i\lambda} K_m(irx) dr. $$
	We construct for $n \in \ag{N}$, immitating \cite[\S 2.8 (8)]{Er56},
	$$ h_{-1-n}(r) = \frac{(-i)^{n+1}}{n!} \int_0^{\infty} u^n (r+iu)^{i\lambda} K_m(-irx+ux) du, \quad \text{resp.} $$
	$$ h_{-1-n}(r) = \frac{i^{n+1}}{n!} \int_0^{\infty} u^n (r-iu)^{i\lambda} K_m(irx + ux) du. $$
	It is easy to compute, using \cite[(3.2)]{Wo77},
	$$ h_{-1}(0) = \frac{(-i)^{n+1}}{x^{1+n+i\lambda}} e^{-\frac{\pi}{2}\lambda} \frac{\Lambda_m(n+i\lambda)}{n!}, \quad \text{resp.} \quad h_{-1}(0) = \frac{i^{n+1}}{x^{1+n+i\lambda}} e^{\frac{\pi}{2}\lambda} \frac{\Lambda_m(n+i\lambda)}{n!}. $$
	Estimating $(r \pm iu)^{i\lambda}$ as in the proof of Lemma \ref{FourErd} and applying Lemma \ref{BesselKLargeBd}, we get
	$$ \extnorm{ h_{-1-n}(r) } \ll \frac{1}{n! x^{1/2} } \int_0^{\infty} u^n (r_0^2+u^2)^{-1/4} e^{-u(x-\norm[\lambda])} du \leq x^{-1/2} (x-\norm[\lambda])^{-1} $$
	and conclude the proof.
\end{proof}

	\subsection{Whittaker New Form at Complex Place}
	
	The Whittaker new forms at complex place have been obtained in \cite{Po08} with the consideration of differential equations, without $\intL^2$-normalizing factor. We give an alternative approach using integral representation. Let $\F = \ag{C}, \pi = \pi(\mu_1, \mu_2)$. Upon twisting by an unramified character we may assume $\mu_1(\rho e^{i\alpha}) = \rho^{i\tau} e^{i n_1 \alpha}, \mu_2(\rho e^{i\alpha}) = \rho^{-i\tau} e^{i n_2 \alpha}$ for some $\tau \in \ag{R}, n_j \in \ag{Z}$. We may assume $n_0:=n_1-n_2 \geq 0$ by exchanging $\mu_1,\mu_2$ if necessary. We have
	$$ \Res_{\SU_2(\ag{C})}^{\GL_2(\ag{C})} \pi = \sideset{}{_{2 \mid n - n_0 \geq 0}} \bigoplus V_n $$
where $V_n$ is the representation of $\SU_2(\ag{C})$ isomorphic to the one $\rho_n$ on the space of homogeneous polynomials $\ag{C}[X,Y]_n$ with two variables of degree $n$. An orthonormal basis of $V_n$ is given by
\begin{align*}
	e_{n,k}(u) &= \sqrt{n+1} \frac{\Pairing{\rho_n(u). X^{n-k}Y^k}{X^{\frac{n+n_0}{2}} Y^{\frac{n-n_0}{2}}}_{\rho_n}}{\Norm[X^{n-k}Y^k]_{\rho_n} \Norm[X^{\frac{n+n_0}{2}} Y^{\frac{n-n_0}{2}}]_{\rho_n}} = Q_{n,k}(\alpha,\beta) \frac{\sqrt{n+1} \Norm[X^{\frac{n+n_0}{2}} Y^{\frac{n-n_0}{2}}]_{\rho_n}}{\Norm[X^{n-k}Y^k]_{\rho_n}} \\
	&= \left( (n+1) \frac{B((n+n_0)/2+1, (n-n_0)/2+1)}{B(n-k+1,k+1)} \right)^{1/2} Q_{n,k}(\alpha,\beta) =: \widetilde{Q}_{n,k}(\alpha, \beta), \quad 0 \leq k \leq n,
\end{align*}
	$$ \text{where} \quad B(x,y) = \frac{\Gamma(x) \Gamma(y)}{\Gamma(x+y)}, \quad u = \begin{pmatrix} \alpha & \beta \\ -\bar{\beta} & \alpha \end{pmatrix} \in \SU_2(\ag{C}). $$
	The polynomials $Q_{n,k}$ satisfying $Q_{n,k}(t\alpha, t\beta) = t^{(n+n_0)/2} \bar{t}^{(n-n_0)/2} Q_{n,k}(\alpha, \beta)$ are in general of complicate form, but are easily determined in the following cases:
	$$ Q_{n_0,k}(\alpha, \beta) = \alpha^{n_0-k} \beta^k, \quad 0 \leq k \leq n_0; $$
	$$ Q_{n,0}(\alpha, \beta) = (-1)^{\frac{n-n_0}{2}} \binom{n}{(n-n_0)/2} \alpha^{\frac{n+n_0}{2}} \bar{\beta}^{\frac{n-n_0}{2}}, \quad Q_{n,n}(\alpha, \beta) = \binom{n}{(n-n_0)/2} \beta^{\frac{n+n_0}{2}} \bar{\alpha}^{\frac{n-n_0}{2}}. $$
	Define $P_{n,k} \in \Sch(\ag{C}^2), f_{n,k} \in \pi $ by
	$$ P_{n,k}(z_1,z_2) := \widetilde{Q}_{n,k}(\bar{z}_2, - \bar{z}_1) e^{-2\pi (\norm[z_1]^2 + \norm[z_2]^2)}, \quad f_{n,k}(g) := \mu_1(\det g) \norm[\det g] \int_{\ag{C}^{\times}} P_{n,k}((0,t)g) \mu_1 \mu_2^{-1}(t) \norm[t]_{\ag{C}} d^{\times}t. $$
	We easily verify that
	$$ f_{n,k}(u) = \Gamma_{\ag{C}}(1+n/2 + i\tau) e_{n,k}(u). $$
	The Whittaker function $W_{n,k}$ of $f_{n,k}$ being determined by
	$$ W_{n,k}(a(y)) = \mu_2(y) \norm[y] \int_{\ag{C}^{\times}} \Four[2]{P_{n,k}}(t, \frac{y}{t}) \mu_1 \mu_2^{-1}(t) d^{\times}t, $$
	we deduce easily that for any $s \in \ag{C}$
\begin{align*}
	\int_{\ag{C}^{\times}} W_{n_0,k}(a(y)) \norm[y]_{\ag{C}}^s d^{\times}y &= \int_{(\ag{C}^{\times})^2} \Four[2]{P_{n_0,k}}(z_1, z_2) \mu_1(z_1) \norm[z_1]_{\ag{C}}^{s+1/2} \mu_2(z_2) \norm[z_2]_{\ag{C}}^{s+1/2} d^{\times}z_1 d^{\times}z_2 \\
	&= i^{3k-n} \int_{(\ag{C}^{\times})^2} \bar{z}_1^k z_2^{n_0-k} e^{-2\pi (\norm[z_1]^2 + \norm[z_2]^2)} \mu_1(z_1) \norm[z_1]_{\ag{C}}^{s+1/2} \mu_2(z_2) \norm[z_2]_{\ag{C}}^{s+1/2} d^{\times}z_1 d^{\times}z_2.
\end{align*}
	In order for the last integral to represent $\Gamma_{\ag{C}}(s+1/2, \mu_1) \Gamma_{\ag{C}}(s+1/2, \mu_2)$, we need $k \leq n_1, n_0-k \leq -n_2$, i.e. $k=n_1 \geq 0 \geq n_2$. Hence an $\intL^2$-normalized Whittaker new form is given by
	$$ W_0 = \Gamma_{\ag{C}}(1+(\norm[n_1]+\norm[n_2])/2 + i\tau)^{-1} W_{\norm[n_1]+\norm[n_2],n_1}. $$
	Similarly, in the case $n_1 \geq n_2 \geq 0$ resp. $0 \geq n_1 \geq n_2$, an $\intL^2$-normalized Whittaker new form is given by
	$$ W_0 = \Gamma_{\ag{C}}(1+(\norm[n_1]+\norm[n_2])/2 + i\tau)^{-1} W_{\norm[n_1]+\norm[n_2],\norm[n_1]+\norm[n_2]} \quad \text{resp.} \quad \Gamma_{\ag{C}}(1+(\norm[n_1]+\norm[n_2])/2 + i\tau)^{-1} W_{\norm[n_1]+\norm[n_2],0}. $$
\begin{proposition}
	Let $\pi=\pi(\mu_1,\mu_2)$ with $\mu_1(\rho e^{i\alpha}) = \rho^{i\tau} e^{i n_1 \alpha}, \mu_2(\rho e^{i\alpha}) = \rho^{-i\tau} e^{i n_2 \alpha}$ for some $\tau \in \ag{R}, n_j \in \ag{Z}$. Assume $n_1 \geq n_2$. A unitary Whittaker new form $W_0$ of $\pi$ is determined by the following conditions:
\begin{itemize}
	\item[(1)] $W_0(a(y)) = W_0(a(\norm[y]))$, i.e., it is a radial function.
	\item[(2)] Let $K_{\nu}(z)$ denote the usual Bessel-$K$ function \cite[\S \Rmnum{4}.6.22 (5)]{Wat44}, then
	$$ W_0(a(y)) = \frac{4 y^{(\norm[n_1]+\norm[n_2])/2+1} K_{(\norm[n_1]-\norm[n_2])/2+i\tau}(4\pi y)}{\Gamma_{\ag{C}}(1+(\norm[n_1]+\norm[n_2])/2+i\tau) \sqrt{B(\norm[n_1]+1, \norm[n_2]+1)} }, \quad y >0. $$
\end{itemize}
	Moreover, we have an integral representation up to a constant of modulus $1$
	$$ W_0(a(y)) = \frac{2 \mu_2(y) \norm[y]}{\sqrt{B(\norm[n_1]+1, \norm[n_2]+1)}} \int_0^{\infty} \int_0^{2\pi} \frac{\rho^{1+\norm[n_2]}}{(1+\rho^2)^{1+(\norm[n_1]+\norm[n_2])/2+i\tau}} e^{-4\pi i y\rho \cos \alpha + i n_2 \alpha} d\alpha d\rho. $$
\label{WhitNewCp}
\end{proposition}
\begin{proof}
	The first part is a summary of the above discussion. If $e_{n,k}$ gives the new vector, then we have for $y>0$ by definition
	$$ W_0(a(y)) = 2 \mu_2(y) \norm[y] \int_{\ag{C}} \frac{\widetilde{Q}_{n,k}(\bar{x}, -1)}{(1+\norm[x]^2)^{1+n/2+i\tau}} e^{-2\pi i y (x+\bar{x})} dx. $$
	Noting that in each case
	$$ \widetilde{Q}_{n,k}(\rho e^{-i\alpha}, -1) = B(\norm[n_1]+1, \norm[n_2]+1)^{-1/2} \rho^{\norm[n_2]} e^{i n_2 \alpha}, $$
	we obtain the formula in the ``moreover'' part.
\end{proof}
\begin{corollary}
	The possible minimal vectors, in the sense of Definition \ref{MinVec}, are $e_{n_0, n_0/2}$ if $2 \mid n_0$ or $e_{n_0, (n_0 \pm 1)/2}$ if $2 \nmid n_0$.
\label{MinVecCp}
\end{corollary}

	\subsection{Refined Sobolev Inequalities}

	Let $\F = \ag{R}$ or $\ag{C}$. We proceed under \emph{Assumption (B)}.
\begin{lemma}
	Notations are the same as in Lemma \ref{KirMaxLowerBd}. If $\pi$ is principal series and $W_0$ is the Kirillov function of a minimal vector, then we have, uniformly in $\Cond(\pi)$ and for $0 < \norm[y]_v \ll \Cond(\pi)^{1/4}$,
	$$ \extnorm{ W_0(y) } \ll \norm[y]_v^{1/2} (1+\extnorm{\log \norm[y]_v}). $$
\label{NVKir0Bd}
\end{lemma}
\begin{proof}
	For $\F = \ag{R}$, by twisting, we may assume $\omega=1$ or $\sgn$ and $\pi = \pi(\norm^{i\tau/2}, \norm^{-i\tau/2})$ resp. $\pi = \pi(\norm^{i\tau/2}\sgn, \norm^{-i\tau/2})$ for some $\tau \in \ag{R}$. Thus
\begin{align}
	W_0(a(y)) &= \frac{\pi^{1/2} \Gamma(i\tau/2)}{2 \Gamma((1+i\tau)/2)} \norm[y]^{(1-i\tau)/2} \sum_{n=0}^{\infty} \frac{(\pi y)^{2n}}{n! (1-\frac{i\tau}{2}) \cdots (n-\frac{i\tau}{2})} \nonumber \\
	&\quad + \frac{\pi^{(1+2i\tau)/2} \Gamma(-i\tau/2)}{2 \Gamma((1+i\tau)/2)} \norm[y]^{(1+i\tau)/2} \sum_{n=0}^{\infty} \frac{(\pi y)^{2n}}{n! (1+\frac{i\tau}{2}) \cdots (n+\frac{i\tau}{2})}; \quad \text{resp.} \label{NVKir0ExpReal1}
\end{align}
\begin{align}
	W_0(a(y)) &= \frac{i \pi^{1/2} \Gamma((1+i\tau)/2)}{\Gamma((2+i\tau)/2)} \norm[y]^{(1-i\tau)/2} \sum_{n=0}^{\infty} \frac{(\pi y)^{2n}}{n! (1-\frac{1+i\tau}{2}) \cdots (n-\frac{1+i\tau}{2})} \nonumber \\
	&\quad - \frac{i\pi^{(1+2i\tau)/2} \Gamma((1-i\tau)/2)}{\Gamma((2+i\tau)/2)} \norm[y]^{(1+i\tau)/2} \sum_{n=0}^{\infty} \frac{(\pi y)^{2n}}{n! (1-\frac{1-i\tau}{2}) \cdots (n-\frac{1-i\tau}{2})}. \label{NVKir0ExpReal2}
\end{align}
	These formulas are classical and can be obtained by the expansion of Bessel-K functions at the origin, for example. They give good estimation for $\norm[y] \ll (1+\norm[\tau])^{1/2}$. The inclusion of ``$\log \norm[y]$'' is only necessary for $\tau=0$.
	
\noindent For $\F = \ag{C}$, by twisting by an unramified character, we may assume $\pi = \pi(\mu_1, \mu_2)$ with $\mu_1(\rho e^{i\alpha}) = \rho^{i\tau} e^{i n_1 \alpha}, \mu_2(\rho e^{i\alpha}) = \rho^{-i\tau} e^{i n_1 \alpha}$ for some $\tau \in \ag{R}, n_1 \in \ag{N}$, since we are under \emph{Assumption (B)}. Corollary \ref{MinVecCp} implies that $W_0(ye^{i\alpha}) = W_0(y) e^{in_1\alpha}$ and for $y>0$
\begin{align}
	W_0(a(y)) &= \frac{2\pi \Gamma(i\tau)}{2\Gamma(1+i\tau)} y^{1-i\tau} \sum_{n=0}^{\infty} \frac{(2\pi y)^{2n}}{n! (1-i\tau) \cdots (n-i\tau)} \nonumber \\
	&\quad + \frac{(2\pi)^{1+2i\tau} \Gamma(i\tau)}{2\Gamma(1+i\tau)} y^{1+i\tau} \sum_{n=0}^{\infty} \frac{(2\pi y)^{2n}}{n! (1+i\tau) \cdots (n+i\tau)}.
\label{NVKir0ExpCp}
\end{align}
	We conclude as in the real case.
\end{proof}
\begin{lemma}
	Let $\Sob^{\gp{K}}$ be a Sobolev norm system defined by the differential operators on $\gp{K}$ ($\Sob_d^{\gp{K}}(\cdot) = \Norm[(1+\Casimir_{\gp{K}})^{d/2} \cdot]$ where $\Casimir_{\gp{K}}$ is the positive Casimir operator of $\gp{K}$, for example). If $\pi$ is principal series and $W \in \pi^{\infty}$ is a smooth vector in the Kirillov model, then for $0 < \norm[y]_v \ll \Cond(\pi)^{1/4}$, we have
	$$ \extnorm{ W(y) } \ll \norm[y]_v^{1/2} (1+\extnorm{\log \norm[y]_v}) \Sob_3^{\gp{K}}(W). $$
\label{SVKir0Bd}
\end{lemma}
\begin{proof}
	Let $W=W_f$ be associated with $f \in \pi^{\infty}$ in the induced model. For $\F = \ag{R}$, by twisting, we may assume $\pi = \pi(\norm^{i\tau/2}, \norm^{-i\tau/2})$ resp. $\pi = \pi(\norm^{i\tau/2}\sgn, \norm^{-i\tau/2})$ for some $\tau \in \ag{R}$. We treat the second case, the first one being simpler. Defining and writing
	$$ \tilde{f}(\alpha) := f(\begin{pmatrix} \cos \alpha & - \sin \alpha \\ \sin \alpha & \cos \alpha \end{pmatrix}) - f(1)e^{i\alpha}, $$
	$$ W_f(y) = \norm[y]_v^{(1-i\tau)/2} \int_0^{\pi} \tilde{f}(\alpha) (\sin^2 \alpha)^{\frac{i\tau - 1}{2}} e^{-2\pi i y \frac{\cos \alpha}{\sin \alpha}} d\alpha + f(1) W_0(a(y)), $$
we easily conclude by Lemma \ref{NVKir0Bd} and
	$$ \extnorm{ \int_0^{\pi} \tilde{f}(\alpha) (\sin^2 \alpha)^{\frac{i\tau - 1}{2}} e^{-2\pi i y \frac{\cos \alpha}{\sin \alpha}} d\alpha } \leq \sup_{0 \leq \alpha \leq \pi} \norm[\tilde{f}'(\alpha)] \cdot 2\int_0^{\pi/2} \frac{\norm[\alpha]}{\norm[\sin \alpha]} d \alpha, $$
	$$ \sup_{0 \leq \alpha \leq \pi} \norm[\tilde{f}'(\alpha)] \ll \Sob_2^{\gp{K}}(f), \quad \norm[f(1)] \ll \Sob_1^{\gp{K}}(f). $$
For $\F = \ag{C}$, by twisting by an unramified character, we may assume $\pi = \pi(\mu_1, \mu_2)$ with $\mu_1(\rho e^{i\alpha}) = \rho^{i\tau} e^{i n_1 \alpha}, \mu_2(\rho e^{i\alpha}) = \rho^{-i\tau} e^{i n_1 \alpha}$ for some $\tau \in \ag{R}, n_1 \in \ag{N}$, since we are under \emph{Assumption (B)}. Defining and writing for $y>0$
	$$ \tilde{f}(\alpha, \beta) := f(\begin{pmatrix} e^{i\alpha} \cos \beta & - \sin \beta \\ \sin \beta & e^{-i\alpha} \cos \beta \end{pmatrix}) - f(1), $$
	$$ W_f(ye^{i\theta}) = y^{1-i\tau} e^{in_1 \theta} \int_0^{\pi/2} \int_0^{2\pi} \tilde{f}(\alpha, \beta) (\sin \beta)^{2i\tau - 1} (\cos \beta) e^{-4\pi i y \frac{\cos \beta}{\sin \beta} \cos(\alpha+\theta)}  d\alpha d\beta + f(1) W_0(a(y)), $$
we easily conclude by Lemma \ref{NVKir0Bd} and
	$$ \extnorm{ \int_0^{\pi/2} \int_0^{2\pi} \tilde{f}(\alpha, \beta) (\sin \beta)^{i\tau - 1} (\cos \beta) e^{-4\pi i y \frac{\cos \beta}{\sin \beta} \cos(\alpha+\theta)}  d\alpha d\beta } \leq \sup_{\substack{ 0 \leq \alpha \leq 2\pi \\ 0 \leq \beta \leq \pi/2 }} \extnorm{\frac{\partial}{\partial \beta} \tilde{f}(\alpha,\beta)} \cdot 2\pi \int_0^{\pi/2} \frac{\beta \cos \beta}{\sin \beta} d\beta, $$
	$$ \sup_{\substack{ 0 \leq \alpha \leq 2\pi \\ 0 \leq \beta \leq \pi/2 }} \extnorm{\frac{\partial}{\partial \beta} \tilde{f}(\alpha,\beta)} \ll \Sob_3^{\gp{K}}(f), \quad \norm[f(1)] \ll \Sob_2^{\gp{K}}(f). $$
	In fact, to obtain the last inequalities, it suffices to decompose $f$ in terms of $e_{n,k}$ using Fourier inversion on $\SU_2(\ag{C})$, notice that
	$$ \frac{\partial}{\partial \beta} = \begin{pmatrix} e^{-i\alpha/2} & \\ & e^{i\alpha/2} \end{pmatrix} \begin{bmatrix} & -1 \\ 1 & \end{bmatrix} \begin{pmatrix} e^{i\alpha/2} & \\ & e^{-i\alpha/2} \end{pmatrix} = \begin{bmatrix} & -e^{-i\alpha} \\ e^{i\alpha} & \end{bmatrix} $$
	as element in the Lie algebra, take into account the formula of actions of ``$X_{\pm}$'' given in \cite[\S 2.7.2]{Wu14}, and the obvious bound
	$$ \norm[ e_{n,k}(u) ] \leq \sqrt{n+1}, \forall u \in \SU_2(\ag{C}). $$
\end{proof}
\begin{lemma}
	Notations are as in Lemma \ref{SVKir0Bd}. Identify the elements of the Lie algebra with the differential operators in the Kirillov model of
	$$ U = \begin{bmatrix} 0 & 1 \\ 0 & 0 \end{bmatrix}; \quad \bar{U}= \begin{bmatrix} 0 & i \\ 0 & 0 \end{bmatrix} \quad \text{if} \quad \F=\ag{C}. $$
	Then for any $\epsilon > 0$ we have
	$$ \Norm[W]_1 \ll_{\epsilon} \left\{ \begin{matrix} (\Norm[U.W]_2^{\epsilon} + \Sob_3^{\gp{K}}(W)^{\epsilon}) \cdot \Norm[W]_2^{1-\epsilon} & \text{if } \F = \ag{R}; \\ (\Norm[(U^2+\bar{U}^2).W]_2^{\epsilon} + \Sob_3^{\gp{K}}(W)^{\epsilon}) \cdot \Norm[W]_2^{1-\epsilon} & \text{if } \F = \ag{C}. \end{matrix} \right. $$
\label{SVKirL1Bd}
\end{lemma}
\begin{proof}
	For $\F=\ag{R}$ resp. $\F=\ag{C}$, we have
	$$ U.W(y) = 2\pi i y W(y) \quad \text{resp.} \quad (U^2+\bar{U}^2).W(y) = -16\pi^2 \norm[y]_{\ag{C}} W(y). $$
	The bound then follows easily from
	$$ \int_{\norm[y]_v \geq 1} \norm[W(y)] d^{\times}y \leq \left( \int_{\norm[y]_v \geq 1} \norm[W(y)]^2 \norm[y]_v^2 d^{\times}y \right)^{\frac{\epsilon}{2}} \left( \int_{\norm[y]_v \geq 1} \norm[W(y)]^2 d^{\times}y \right)^{\frac{1-\epsilon}{2}} \left( \int_{\norm[y] \geq 1} \norm[y]_v^{-2\epsilon} d^{\times}y \right)^{\frac{1}{2}}, $$
	$$ \int_{\norm[y]_v \leq 1} \norm[W(y)] d^{\times}y \ll \Sob_3^{\gp{K}}(W)^{\epsilon} \left( \int_{\norm[y]_v \leq 1} \norm[W(y)]^2 d^{\times}y \right)^{\frac{1-\epsilon}{2}} \left( \int_{\norm[y]_v \leq 1} \norm[y]_v^{\frac{\epsilon}{1+\epsilon}} d^{\times}y \right)^{\frac{1+\epsilon}{2}}, $$
	where in the last inequality we have applied Lemma \ref{SVKir0Bd}.
\end{proof}
\begin{lemma}
	Notations are as in Lemma \ref{SVKirL1Bd}. If $-1/2 \leq \sigma < 0$ and $\epsilon > 0$ such that $\sigma + \epsilon < 0$, then
	$$ \int_{\F^{\times}} \norm[W(y)] \norm[y]_v^{\sigma+\epsilon} d^{\times}y \ll_{\sigma, \epsilon} \Sob_3^{\gp{K}}(W)^{-2\sigma} \Norm[W]_2^{1+2\sigma} + \Norm[W]_2. $$
	If for some $n \in \ag{N}$, $n \leq \sigma < n+1$, then
	$$ \int_{\F^{\times}} \norm[W(y)] \norm[y]_v^{\sigma} d^{\times}y \ll_{\sigma} \left\{ \begin{matrix} \Norm[U^n.W]_1^{n+1-\sigma} \Norm[U^{n+1}.W]_1^{\sigma-n} & \text{if } \F = \ag{R}; \\ \Norm[(U^2+\bar{U}^2)^n.W]_1^{n+1-\sigma} \Norm[(U^2+\bar{U}^2)^{n+1}.W]_1^{\sigma-n} & \text{if } \F = \ag{C}. \end{matrix} \right. $$
\label{SVKirL1ShBd}
\end{lemma}
\begin{proof}
	The first inequality follows from
\begin{align*}
	\int_{\norm[y]_v \leq 1} \norm[W(y)] \norm[y]_v^{\sigma+\epsilon} d^{\times}y &\ll \Sob_3^{\gp{K}}(W)^{-2\sigma} \int_{\norm[y]_v \leq 1} \norm[W(y)]^{1+2\sigma} \norm[y]_v^{\epsilon} (1+\extnorm{\log \norm[y]_v})^{-2\sigma} d^{\times}y \\
	&\ll_{\sigma, \epsilon} \Sob_3^{\gp{K}}(W)^{-2\sigma} \left( \int_{\norm[y]_v \leq 1} \norm[W(y)]^2 d^{\times}y \right)^{(1+2\sigma)/2},
\end{align*}
	$$ \int_{\norm[y]_v \geq 1} \norm[W(y)] \norm[y]_v^{\sigma+\epsilon} d^{\times}y \leq \left( \int_{\norm[y]_v \geq 1} \norm[W(y)]^2 d^{\times}y \right)^{1/2} \left( \int_{\norm[y]_v \geq 1} \norm[y]_v^{2\sigma+2\epsilon} d^{\times}y \right)^{1/2}. $$
	The second one follows from a standard interpolation argument.
\end{proof}
\begin{lemma}
	Let $\pi=\pi(\norm_v^{i\tau}, \norm_v^{-i\tau})$ and $W$ be the Kirillov function of a $\gp{K}_v$-isotypic vector of $\pi$. Write
	$$ W(y) = a_+(W) \norm[y]_v^{1/2+i\tau} + a_-(W) \norm[y]_v^{1/2-i\tau} + \widetilde{W}(y), $$
	where $a_{\pm}(W) \in \ag{C}$ are so defined that
	$$ \extnorm{ \widetilde{W}(y) } = o(\norm[y]_v^{1/2}), \quad y \to 0. $$
	Let $\Delta_v$ be the Laplacian, local component of $\Delta_{\infty}$ defined in Lemma \ref{LindelAvg}. Then we have for $y>0$
	$$ \extnorm{ a_{\pm}(W) } \ll_{\epsilon} \norm[\tau]^{-1/2} \Norm[W]_2^{1/2+\epsilon} \Norm[\Delta_v^{1/2}.W]_2^{1/2+\epsilon}; $$
	$$ \extnorm{ \widetilde{W}(y) } \ll_{\epsilon} \Norm[\Delta_v.W]_2^{1/2+\epsilon} \Norm[\Delta_v^{3/2}.W]_2^{1/2+\epsilon} \norm[y]_v^{1+\epsilon}, \quad \extnorm{ \frac{d}{dy} \widetilde{W}(y) } \ll_{\epsilon} \norm[\tau]^{1/2}\Norm[\Delta_v.W]_2^{1/2+\epsilon} \Norm[\Delta_v^{3/2}.W]_2^{1/2+\epsilon} \norm[y]_v^{\epsilon}. $$
\label{KIKirBd}
\end{lemma}
\begin{proof}
	This is a refinement of \cite[Proposition 3.2.3]{MV10} in a special case. We first consider the real case. Applying Mellin inversion and local functional equation, we get for $m \in \{0,1\}, y >0$
\begin{align*}
	W(y) + (-1)^m W(-y) &= \gamma^*(-i\tau-m,\sgn^m) \zeta(1+i\tau+m, w.W, \sgn^m) y^{1/2+i\tau+m} \\
	&\quad + \gamma^*(i\tau-m,\sgn^m) \zeta(1-i\tau+m, w.W, \sgn^m) y^{1/2-i\tau+m} \\
	&\quad + \int_{\Re s = -1-m-\epsilon} \gamma(1/2+s,\sgn^m) \zeta(1/2-s,w.W,\sgn^m) y^{-s} \frac{ds}{2\pi i},
\end{align*}
	where the gamma factor
	$$ \gamma(1/2+s,\sgn^m) = \varepsilon_m \pi^{-2s} \frac{\Gamma((1/2+s+i\tau+m)/2) \Gamma((1/2+s-i\tau+m)/2)}{\Gamma((1/2-s+i\tau+m)/2) \Gamma((1/2-s-i\tau+m)/2)}, $$
	$\gamma^*$ is the residue and $\norm[\varepsilon_m] = 1$ \cite[\S 7.1]{RV99}. Thus
	$$ a_+(W) = \gamma^*(-i\tau,1) \zeta(1+i\tau, w.W, 1)/2, \quad a_-(W) = \gamma^*(i\tau,1) \zeta(1-i\tau, w.W, 1)/2. $$
	Stirling's formula implies
	$$ \extnorm{ \gamma^*(\pm i\tau - m,\sgn^m) } \ll \norm[\tau]^{-1/2}, \quad \extnorm{ \gamma(1/2+s,\sgn^m) } \ll_{\epsilon} \norm[ \Im s + \tau ]^{-1-m-\epsilon} \norm[ \Im s - \tau ]^{-1-m-\epsilon}. $$
	The zeta integrals admit trivial bounds
	$$ \extnorm{ \zeta(1\pm i\tau+m, w.W, \sgn^m) } \leq \int_{\ag{R}^{\times}} \norm[w.W(y)] \norm[y]^{1/2+m} d^{\times}y, $$
	$$ \extnorm{ \zeta(1/2-s,w.W,\sgn^m) } \leq \int_{\ag{R}^{\times}} \norm[w.W(y)] \norm[y]^{1+m+\epsilon} d^{\times}y. $$
	We conclude the proof by applying Lemma \ref{SVKirL1ShBd} and \cite[\S 2.7.1]{Wu14}.
	
\noindent In the complex case, write $\sgn(re^{i\alpha}) = e^{i\alpha}$ for $r > 0, \alpha \in \ag{R}/2\pi \ag{Z}$. There is a unique $m \in \ag{Z}$ such that $W(re^{i\alpha})=W(r)e^{-im\alpha}$. Hence
	$$ W(r) = \frac{1}{2\pi} \int_0^{2\pi} W(re^{i\alpha}) e^{im\alpha} d\alpha = \frac{1}{4} \int_{\Re s \gg 1} \zeta(s+1/2,W,\sgn^m) r^{-s} \frac{ds}{2\pi i}. $$
	We then argue as in the real case by shifting the contour to $\Re s = -1-m-\epsilon$.
\end{proof}

\section*{Acknowledgement}
	
	The preparation of the paper scattered during the stays of the author's in FIM at ETHZ, at Alf\'ed Renyi Institute in Hungary supported by the MTA R\'enyi Int\'ezet Lend\"ulet Automorphic Research Group and in TAN at EPFL. The author would like to thank all the three institutes for their hospitality.

\bibliographystyle{acm}
	
\bibliography{mathbib}	
	
\address{\quad \\ Han WU \\ EPFL SB MATHGEOM TAN \\ MA C3 604 \\ Station 8 \\ CH-1015, Lausanne \\ Switzerland \\ wuhan1121@yahoo.com}
	
\end{document}